\documentclass[a4paper,10pt]{article}
\usepackage{latexsym,xypic}
\usepackage{amssymb}
\usepackage{amsmath}
\usepackage{amsfonts}

\newtheorem{theorem}{Theorem}

\newtheorem{definition}[theorem]{Definition}
\newtheorem{example}[theorem]{Example}

\newtheorem{lemma}[theorem]{Lemma}

\newtheorem{proposition}[theorem]{Proposition}
\newtheorem{remark}[theorem]{Remark}

\newenvironment{proof}{{\bf Proof:} }{$\Box$\mbox{}}

\begin{document}
\begin{titlepage}
  \title
  {\bf THREE CROSSED MODULES}

  \author
  { Z. ARVAS\.{I}, T. S. KUZPINARI and E. \"{O}. USLU   \\
   \mbox{} \\
 }
\date{}
\end{titlepage}

\maketitle

\begin{abstract}
We introduce the notion of 3-crossed module, which extends the
notions of 1-crossed module (Whitehead) and 2-crossed module
(Conduch\'{e}). We show that the category of 3-crossed modules is
equivalent to the category of simplicial groups having a Moore
complex of length 3. We make explicit the relationship with the
cat$^{3}$-groups (Loday) and the 3-hyper-complexes
(Cegarra-Carrasco), which also model algebraically homotopy 4-types.
\end{abstract}
\begin{center}
\begin{tabular}{llllll}
{\tt Keywords:} & Crossed module, 2-crossed module,  \\
                & Simplicial group, Moore complex.\\
{\tt A. M. S. C.:} & 18D35 18G30 18G50 18G55.
\end{tabular}
\end{center}
\maketitle
\section{Introduction}

Crossed modules (or 1-crossed modules) were first defined by
Whitehead in
\cite{jh1}. They model connected homotopy $2$-types. Conduch\'{e} \cite%
{conduche} in 1984 described the notion of $2$-crossed module as a
model of connected $3$-types. More generally Loday, \cite{jl}, gave
the foundation of a theory of another algebraic model, which is
called cat$^{n}$-groups, for connected $(n+1)$-types. Ellis-Steiner
\cite{es} shown that cat$^{n}$-groups are equivalent to crossed
$n$-cubes. A link between simplicial groups and crossed $n$-cubes
were given by Porter \cite{p6}. Conduch\'{e} \cite{cond}
gives a relation between crossed $2$-cubes (i.e. crossed squares) and $2$%
-crossed modules. $2$-crossed modules were known to be equivalent to
that of simplicial groups whose Moore complex has length $2.$ In
\cite{baues,baues2} Baues introduced a related notion of quadratic
module. The first author and Ulualan \cite{zarvasi} also explored
some relations among these algebraic models for (connected) homotopy
$3$-types.

The most general investigation into the extra structure of the Moore
complex of a simplicial group was given by Carrasco-Cegarra in
\cite{c} to construct the Non-Abelian version of the classical
Dold-Kan theorem. A much more general context of their work was
given by Bourn in \cite{bourn}. Carrasco and Cegarra arrived at a
notion of a hypercrossed complexes and proved that the category of
such hypercrossed complexes is equivalent to that of simplicial
groups. If one truncates hypercrossed complexes at level $n$,
throwing away terms of higher dimension, the resulting
$n$-hypercrossed complexes form a category equivalent to the
$n$-hyper groupoids of groups given by Duskin, \cite{duskin}, Glenn
\cite{glenn} and give algebraic models for $n$-types. For $n=1$, a
$1$-hypercrossed complex gives a crossed module, whilst a
subcategory of the category of hypercrossed $2$-complexes is
equivalent to Conduche's category of $2$-crossed modules.

Mutlu-Porter, \cite{amut4}, introduced a Peiffer pairing structure
within the Moore complexes of a simplicial group. They applied this
structure to the study of algebraic models for homotopy types.

In this article we will define the notion of $3$-crossed module as a
model for homotopy $4$-types. The methods we use are based on ideas
of Conduch\'{e} given in \cite{conduche} and a Peiffer pairing
structure within the Moore complexes of a simplicial group. We prove
that the category of 3-crossed modules is equivalent to that of
simplicial groups with Moore complex of length 3 which is equivalent
to that of 3-hypercrossed complexes. The main problem with the
$3$-hypercrossed complex is difficult to handle intuitively.

The advantage of the notion of $3$-crossed module is the following;

\begin{enumerate}
\item[(i)] A new algebraic model for (connected) homotopy $4$-types;

\item[(ii)] It is easy to handle with respect to other models such as the $3$%
-hypercrossed complex;

\item[(iii)] Give a possible way  to generalising $n$-crossed modules
(or equivalently $n$-groups (see \cite{robert} )) which is analogues to a $n$%
-hypercrossed complex.

\item[(iv)] In \cite{baues2}, Baues points out that a
\textquotedblleft nilpotent\textquotedblright algebraic model for
4-types is not known. $3$-crossed modules go some way towards that
aim.
\end{enumerate}

\section{Simplicial groups, Moore Complexes, Peiffer pairings}

We refer the reader to \cite{may} and  \cite{curtis} for the basic
properties of simplicial structures.

\subsection{\textbf{Simplicial Groups }}

A simplicial group $\mathbf{G}$ consists of a family of groups
$\left\{
G_{n}\right\} $ together with face and degeneracy maps $d_{i}^{n}:G_{n}%
\rightarrow G_{n-1}$, $0\leq i\leq n$, $(n\neq 0)$ and $s_{i}^{n}:G_{n-1}%
\rightarrow G_{n},0\leq i\leq n$, satisfying the usual simplicial
identities given in \cite{may}, \cite{curtis}. The category of
simplicial groups will be denoted by $\mathbf{SimpGrp}$.

Let $\Delta $ denotes the category of finite ordinals. For each
$k\geq 0$ we
obtain a subcategory $\Delta _{\leq k}$ determined by the objects $\left[ i%
\right] $ of $\Delta $ with $i\leq k$. A $k$-truncated simplicial
group is a functor from $\Delta _{\leq k}^{op}$ to $\mathbf{Grp}$
(the category of
groups). We will denote the category of $k$-truncated simplicial groups by $%
\mathbf{Tr}_{k}\mathbf{SimpGrp}\mathfrak{.}$ By a $k$-$truncation$ $of$ $a$ $%
simplicial$ $group,$ we mean a $k$-truncated simplicial group $\mathbf{tr}%
_{k}\mathbf{G}$ obtained by forgetting dimensions of order $>k$ in a
simplicial group $\mathbf{G.}$ Then we have the adjoint situation

\begin{center}
\begin{center}
$\xymatrix{  \mathbf{SimpGrp} \ar@{->}@<2pt>[rr]^-{\mathbf{tr}_{k}}
                 & &\ar@{->}@<2pt>[ll]^-{\mathbf{st}_{k}}
 \mathbf{Tr}_{k} \mathbf{SimpGrp}
  }$
\end{center}%
\end{center}

\noindent where $\mathbf{st}_{k}$ is called the $k$-skeleton
functor. For detailed definitions see \cite{duskin}.

\subsection{\textbf{The Moore Complex}.}

The Moore complex $\mathbf{NG}$\ of a simplicial group $\mathbf{G}$\
is
defined to be the normal chain complex $(\mathbf{NG,\partial })$ with%
\begin{equation*}
NG_{n}=\bigcap\limits_{i=0}^{n-1}Kerd_{i}
\end{equation*}%
and with differential $\partial _{n}:NG_{n}\rightarrow NG_{n-1}$
induced from $d_{n}$ by restriction.

The n$^{th}$ \textit{homotopy group }$\pi _{n}$($\mathbf{G}$) of
$\mathbf{G}$ is the n$^{th}$ homology of the Moore complex of
$\mathbf{G}$, i.e.
\begin{equation*}
\pi _{n}(\mathbf{G})\cong H_{n}(\mathbf{NG},\partial
)=\bigcap_{i=0}^{n}\ker
d_{i}^{n}/d_{n+1}^{n+1}(\bigcap_{i=0}^{n}\ker d_{i}^{n+1}).
\end{equation*}

We say that the Moore complex $\mathbf{NG}$ of a simplicial group
$\mathbf{G} $ is of \textit{length k} if $\mathbf{NG}_{n}=1$ for all
$n\geq k+1$. We denote the category of simplicial groups with Moore
complex of length $k$ by $\mathbf{SimpGrp}_{\leq k}.$

The Moore complex, $\mathbf{NG}$, carries a hypercrossed complex
structure (see Carrasco \cite{c} ) from which $\mathbf{G}$ can be
rebuilt. We recall briefly some of the aspects of this
reconstructions that we will need later.

\subsection{\textbf{The Poset of Surjective Maps}}

The following notation and terminology is derived from \cite{cc}.

For the ordered set $[n]=\{0<1<\dots <n\}$, let $\alpha
_{i}^{n}:[n+1]\rightarrow \lbrack n]$ be the increasing surjective
map given by;
\begin{equation*}
\alpha _{i}^{n}(j)=\left\{
\begin{array}{ll}
j & \text{if }j\leq i, \\
j-1 & \text{if }j>i.%
\end{array}%
\right.
\end{equation*}%
Let $S(n,n-r)$ be the set of all monotone increasing surjective maps from $%
[n]$ to $[n-r]$. This can be generated from the various $\alpha
_{i}^{n}$ by composition. The composition of these generating maps
is subject to the following rule: $\alpha _{j}\alpha _{i}=\alpha
_{i-1}\alpha _{j},j<i$. This
implies that every element $\alpha \in S(n,n-r)$ has a unique expression as $%
\alpha =\alpha _{i_{1}}\circ \alpha _{i_{2}}\circ \dots \circ \alpha
_{i_{r}} $ with $0\leq i_{1}<i_{2}<\dots <i_{r}\leq n-1$, where the indices $%
i_{k}$ are the elements of $[n]$ such that $\{i_{1},\dots
,i_{r}\}=\{i:\alpha (i)=\alpha (i+1)\}$. We thus can identify
$S(n,n-r)$ with the set $\{(i_{r},\dots ,i_{1}):0\leq
i_{1}<i_{2}<\dots <i_{r}\leq n-1\} $. In particular, the single
element of $S(n,n)$, defined by the
identity map on $[n]$, corresponds to the empty 0-tuple ( ) denoted by $%
\emptyset _{n} $. Similarly the only element of $S(n,0)$ is
$(n-1,n-2,\dots ,0)$. For all $n\geq 0$, let
\begin{equation*}
S(n)=\bigcup_{0\leq r\leq n}S(n,n-r).
\end{equation*}%
We say that $\alpha =(i_{r},\dots ,i_{1})<\beta =(j_{s},\dots ,j_{1})$ in $%
S(n)$ if $i_{1}=j_{1},\dots ,i_{k}=j_{k}$ but $i_{k+1}>j_{k+1},$
$(k\geq 0)$ or if $i_{1}=j_{1},\dots ,i_{r}=j_{r}$ and $r<s$. This
makes $S(n)$ an ordered set. For example

\begin{eqnarray*}
S(2) &=&\{\phi _{2}<(1)<(0)<(1,0)\} \\
S(3) &=&\{\phi _{3}<(2)<(1)<(2,1)<(0)<(2,0)<(1,0)<(2,1,0)\} \\
S(4) &=&\{\phi _{4}<(3)<(2)<(3,2)<(1)<(3,1)<(2,1)<(3,2,1)<(0)<(3,0)<(2,0) \\
&<&(3,2,0)<(1,0)<(3,1,0)<(2,1,0)<(3,2,1,0)\}
\end{eqnarray*}

\subsection{The Semidirect Decomposition of a Simplicial Group}

The fundamental idea behind this can be found in Conduch\'{e} \cite{conduche}%
. A detailed investigation of this for the case of simplicial groups
is given in Carrasco and Cegarra \cite{c}.

Given a split extension of groups

\begin{center}
$\xymatrix@R=40pt@C=40pt{
  1 \ar[r] & K \ar[r] & G \ar@{->}@<3pt>[r]^-{d} & P \ar@{->}@<3pt>[l]^-{s} \ar[r] & 1     }$
\end{center}

\noindent we write $G\cong K\rtimes s(P)$, the semidirect product of
the normal subgroup, $K$, with the image of $P$ under the splitting
$s$.

\begin{proposition}
If \textbf{G} is a simplicial group, then for any $n\geq 0$%
\begin{equation*}
\begin{array}{lll}
G_n & \cong & (\ldots (NG_n \rtimes s_{n-1}NG_{n-1})\rtimes \ldots
\rtimes
s_{n-2}\ldots s_0NG_1)\rtimes \\
&  & \qquad (\ldots (s_{n-2}NG_{n-1}\rtimes
s_{n-1}s_{n-2}NG_{n-2})\rtimes
\ldots \rtimes s_{n-1}s_{n-2}\dots s_0NG_0). ~~%
\end{array}%
\end{equation*}
\end{proposition}

\begin{proof}
This is done by repeatedly use of the following lemma.
\end{proof}

\begin{lemma}
Let \textbf{G} be a simplicial group. Then $G_{n}$ can be decomposed
as a semidirect product:
\begin{equation*}
G_{n}\cong \mathrm{Ker}d_{n}^{n}\rtimes s_{n-1}^{n-1}(G_{n-1}).
\end{equation*}
\end{lemma}

The bracketing and the order of terms in this multiple semidirect
product are generated by the sequence:

\begin{equation*}
\begin{array}{lll}
G_1 & \cong & NG_1 \rtimes s_0NG_0 \\
G_2 & \cong & (NG_2 \rtimes s_1NG_1)\rtimes (s_0NG_1\rtimes s_1s_0NG_0) \\
G_3 & \cong & ((NG_3 \rtimes s_2NG_2)\rtimes (s_1NG_2\rtimes
s_2s_1NG_1))\rtimes \\
&  & \qquad \qquad \qquad ((s_0NG_2\rtimes s_2s_0NG_1)\rtimes
(s_1s_0NG_1\rtimes s_2s_1s_0NG_0)).%
\end{array}%
\end{equation*}
and%
\begin{equation*}
\begin{array}{lll}
G_4 & \cong & (((NG_4 \rtimes s_3NG_3)\rtimes (s_2NG_3\rtimes
s_3s_2NG_2))\rtimes \\
&  & \qquad \ ((s_1NG_3 \rtimes s_3s_1NG_2)
\rtimes(s_2s_1NG_2\rtimes
s_3s_2s_1NG_1)))\rtimes \\
&  & \qquad \qquad s_0( \mbox{\rm decomposition of }G_3).%
\end{array}%
\end{equation*}

Note that the term corresponding to $\alpha =(i_r,\ldots ,i_1)\in S(n)$ is%
\begin{equation*}
s_\alpha (NG_{n-\#\alpha })=s_{i_r...i_1}(NG_{n-\#\alpha
})=s_{i_r}...s_{i_1}(NG_{n-\#\alpha }),
\end{equation*}
where $\#\alpha =r.$ Hence any element $x\in G_n$ can be written in the form%
\begin{equation*}
x=y\prod\limits_{\alpha \in S(n)}s_\alpha (x_\alpha )\mbox{\rm  \qquad with }%
y\in NG_n\mbox{\rm \ and }x_\alpha \in NG_{n-\#\alpha }.
\end{equation*}

\subsection{\textbf{Hypercrossed Complex Pairings}}

In the following we recall from \cite{amut4} hypercrossed complex
pairings. The fundamental idea behind this can be found in Carrasco
and Cegarra (cf.
\cite{c}). The construction depends on a variety of sources, mainly Conduch%
\'{e} \cite{conduche}, Mutlu and Porter \cite{amut4}. Define a set
$P(n)$ consisting of pairs of elements $(\alpha ,\beta )$ from
$S(n)$ with $\alpha \cap \beta =\emptyset $ and $\beta <\alpha $ ,
with respect to lexicographic ordering in $S(n)$ where $\alpha
=(i_{r},\dots ,i_{1}),\beta =(j_{s},\dots ,j_{1})\in S(n)$. The
pairings that we will need,
\begin{equation*}
\{F_{\alpha ,\beta }:NG_{n-\sharp \alpha }\times NG_{n-\sharp \beta
}\rightarrow NG_{n}:(\alpha ,\beta )\in P(n),n\geq 0\}
\end{equation*}%
are given as composites by the diagram

\begin{center}
$\xymatrix@R=40pt@C=60pt{\ NG_{n-\#\alpha}\times NG_{n-\#\beta} \ar[d]%
_{s_{\alpha}\times s_{\beta}} \ar[r]^-{F_{\alpha ,\beta}} & NG_n \\
G_n \times G_n \ar[r]_{\mu} & G_n \ar[u]_{p} }$
\end{center}

\noindent where $s_{\alpha }=s_{i_{r}},\dots ,s_{i_{1}}:NG_{n-\sharp
\alpha }\rightarrow G_{n},$\quad $s_{\beta }=s_{j_{s}},\dots
,s_{j_{1}}:NG_{n-\sharp \beta }\rightarrow G_{n},$
$p:G_{n}\rightarrow NG_{n} $ is defined by composite projections
$p(x)=p_{n-1}\dots p_{0}(x),$ where $p_{j}(z)=zs_{j}d_{j}(z)^{-1}$
with $j=0,1,\dots ,n-1.$ $\mu :G_{n}\times G_{n}\rightarrow G_{n}$
is given by commutator map and $\sharp \alpha $ is the number of the
elements in the set of $\alpha ,$ similarly
for $\sharp \beta .$ Thus%
\begin{eqnarray*}
F_{\alpha ,\beta }(x_{\alpha },y_{\beta }) &=&p\mu \lbrack
(s_{\alpha
}\times s_{\beta })(x_{\alpha },x_{\beta })] \\
&=&p[(s_{\alpha }x_{\alpha }\times s_{\beta }x_{\beta })]
\end{eqnarray*}

Let $N_{n}$ be the normal subgroup of $G_{n}$ generated by elements
of the form
\begin{equation*}
F_{\alpha ,\beta }(x_{\alpha },y_{\beta })
\end{equation*}%
where $x_{\alpha }\in NG_{n-\sharp \alpha }$ and $y_{\beta }\in
NG_{n-\sharp \beta }.$

We illustrate this subgroup for $n=3$ and $n=4$ as follows:

For $n=3$, the possible Peiffer pairings are the following

\begin{center}
$F_{(1,0)(2)}$, $F_{(2,0)(1)}$, $F_{(0)(2,1)}$, $F_{(0)(2)}$, $F_{(1)(2)}$, $%
F_{(0)(1)}$
\end{center}

For all $x_{1}\in NG_{1},y_{2}\in NG_{2},$ the corresponding generators of $%
N_{3}$ are:
\begin{align*}
F_{(1,0)(2)}(x_{1},y_{2})
&=[s_{1}s_{0}x_{1},s_{2}y_{2}][s_{2}y_{2},s_{2}s_{0}x_{1}], \\
F_{(2,0)(1)}(x_{1},y_{2})
&=[s_{2}s_{0}x_{1},s_{1}y_{2}][s_{1}y_{2},s_{2}s_{1}x_{1}][s_{2}s_{1}x_{1},s_{2}y_{2}][s_{2}y_{2},s_{2}s_{0}x_{1}]
\\
\intertext{ and for all $x_{2}\in NG_{2},y_{1}\in NG_{1},$}
F_{(0)(2,1)}(x_{2},y_{1})&=[s_{0}x_{2},s_{2}s_{1}y_{1}][s_{2}s_{1}y_{1},s_{1}x_{2}][s_{2}x_{2},s_{2}s_{1}y_{1}]
\\
\intertext{ whilst for all $x_{2},y_{2}\in NG_{2},$}
F_{(0)(1)}(x_{2},y_{2})
&=[s_{0}x_{2},s_{1}y_{2}][s_{1}y_{2},s_{1}x_{2}][s_{2}x_{2},s_{2}y_{2}], \\
F_{(0)(2)}(x_{2},y_{2}) &=[s_{0}x_{2},s_{2}y_{2}], \\
F_{(1)(2)}(x_{2},y_{2})
&=[s_{1}x_{2},s_{2}y_{2}][s_{2}y_{2},s_{2}x_{2}].
\end{align*}

For $n=4$, the key pairings are thus the following

\begin{center}
\begin{tabular}{lllll}
$F_{(0)(3,2,1)},$ & $F_{(3,2,0)(1)},$ & $F_{(3,1,0)(2)},$ &
$F_{(2,1,0)(3)},$
& $F_{(3,0)(2,1)},$ \\
$F_{(2,0)(3,1)},$ & $F_{(1,0)(3,2)},$ & $F_{(1)(3,2)},$ & $F_{(0)(3,2)},$ & $%
F_{(0)(3,1)},$ \\
$F_{(0)(2,1)},$ & $F_{(3,1)(2)},$ & $F_{(2,1)(3)},$ & $F_{(3,0)(2)},$ & $%
F_{(3,0)(1)},$ \\
$F_{(2,0)(3)},$ & $F_{(2,0)(1)},$ & $F_{(1,0)(3)},$ & $F_{(1,0)(2)},$ & $%
F_{(2)(3)},$ \\
$F_{(1)(3)},$ & $F_{(0)(3)},$ & $F_{(1)(2)},$ & $F_{(0)(2)},$ & $F_{(0)(1)}.$%
\end{tabular}
\end{center}

For $x_1,y_1 \in NG_1$, $x_2,y_2 \in NG_2$ and $x_3,y_3 \in NG_3$
the generator element of the normal subgroup $N_4$ can be easily
written down from Lemma \ref{le25}.

\begin{theorem}
(\cite{amut4})For $n=2,3$ and $4,$ let $\mathbf{G}$ be a simplicial
group with Moore complex $\mathbf{NG}$ in which $G_{n}=D_{n},$ is
the normal subgroup of $G_{n}$ generated by the degenerate elements
in dimension $n,$ then
\begin{equation*}
\begin{tabular}{l}
$\partial _{n}(NG_{n})=\prod\limits_{I,J}\left[ K_{I},K_{J}\right] $%
\end{tabular}%
\end{equation*}%
\ for $I,J\subseteq \lbrack n-1]$ with $I\cup J=[n-1],$
$I=[n-1]-\{\alpha \}$ $J=[n-1]-\{\beta \}$ where $(\alpha ,\beta
)\in P(n).$
\end{theorem}

\begin{remark}
Shortly in \cite{amut4} they defined the normal subgroup $\partial
_{n}(NG_{n}\cap D_{n})$ by $F_{\alpha ,\beta }$ elements which were
defined first by Carrasco in \cite{c}. Castiglioni and Ladra
\cite{ladra} gave a general proof for the inclusions partially
proved by Arvasi and Porter in \cite{patron3}, Arvasi and Ak\c{c}a
in \cite{AA} and Mutlu and Porter in \cite{amut4}. Their approach to
the problem is different from that of cited works. They have
succeeded with a proof, for the case of algebras, over an operad by
introducing a different description of the adjoint inverse of the
normalization functor $\mathbf{N}:Ab^{\Delta ^{op}}\rightarrow
Ch_{\geqslant 0}$, and for the case of groups, they then adapted the
construction for the adjoint inverse used for algebras to get a
simplicial group $G\boxtimes \Lambda $ from the Moore complex of a
simplicial group $G$.
\end{remark}

Following the theorem named as \textbf{Theorem B} in \cite{amut4} we
have

\begin{lemma}
\label{le25} Let $\mathbf{G}$ be a simplicial group with Moore complex $%
\mathbf{NG}$\textbf{\ }of\textbf{\ }length $3$. Then for $n=4$ case
the images of $F_{\alpha ,\beta }$ elements under $\partial _{4}$
given in Table 1 are trivial.
\end{lemma}

\begin{proof}
Since $NG_{4}=1$ by the \textbf{Theorem B} in \cite{amut4} result is
trivial.
\end{proof}

\newpage

\begin{tabular}{|l|l|l|l|}
\hline $1)$ & $d_{4}(F_{(0)(3,2,1)}(x_{3},x_{1}))$ & $=$ & $\left[
s_{0}d_{3}x_{3},s_{2}s_{1}x_{1}\right] \left[ s_{2}s_{1}x_{1},s_{1}d_{3}x_{3}%
\right] $ \\ \hline &  &  & $\left[
s_{2}d_{3}x_{3},s_{2}s_{1}x_{1}\right] \left[
s_{2}s_{1}x_{1},x_{3}\right] $ \\ \hline $2)$ &
$d_{4}(F_{(3,2,0)(1)}(x_{1},x_{3}))$ & $=$ & $\left[
s_{2}s_{0}x_{1},s_{1}d_{3}x_{3}\right] \left[ s_{1}d_{3}x_{3},s_{2}s_{1}x_{1}%
\right] $ \\ \hline &  &  & $\left[
s_{2}s_{1}x_{1},s_{2}d_{3}x_{3}\right] \left[
s_{2}d_{3}x_{3},s_{2}s_{0}x_{1}\right] $ \\ \hline
&  &  & $\left[ s_{2}s_{0}x_{1},x_{3}\right] \left[ x_{3},s_{2}s_{1}x_{1}%
\right] $ \\ \hline $3)$ & $d_{4}(F_{(3,1,0)(2)}(x_{1},x_{3}))$ &
$=$ & $\left[
s_{1}s_{0}x_{1},s_{2}d_{3}x_{3}\right] \left[ s_{2}d_{3}x_{3},s_{2}s_{0}x_{1}%
\right] $ \\ \hline
&  &  & $\left[ s_{2}s_{0}x_{1},x_{3}\right] \left[ x_{3},s_{1}s_{0}x_{1}%
\right] $ \\ \hline $4)$ & $d_{4}(F_{(2,1,0)(3)}(x_{1},x_{3}))$ &
$=$ & $\left[ s_{2}s_{1}s_{0}d_{1}x_{1},x_{3}\right] \left[
x_{3},s_{1}s_{0}x_{1}\right] $
\\ \hline
$5)$ & $d_{4}(F_{(3,0)(2,1)}(x_{2},y_{2}))$ & $=$ & $\left[
s_{0}x_{2},s_{2}s_{1}d_{2}y_{2}\right] \left[ s_{2}s_{1}d_{2}y_{2},s_{1}x_{2}%
\right] $ \\ \hline &  &  & $\left[
s_{2}x_{2},s_{2}s_{1}d_{2}y_{2}\right] \left[
s_{1}y_{2},s_{2}x_{2}\right] $ \\ \hline
&  &  & $\left[ s_{1}x_{2},s_{1}y_{2}\right] \left[ s_{1}y_{2},s_{0}x_{2}%
\right] $ \\ \hline $6)$ & $d_{4}(F_{(2,0)(3,1)}(x_{2},y_{2}))$ &
$=$ & $\left[
s_{2}s_{0}d_{2}x_{2},s_{1}y_{2}\right] \left[ s_{1}y_{2},s_{2}s_{1}d_{2}x_{2}%
\right] $ \\ \hline &  &  & $\left[
s_{2}s_{1}d_{2}x_{2},s_{2}y_{2}\right] \left[
s_{2}y_{2},s_{2}s_{0}d_{2}x_{2}\right] $ \\ \hline
&  &  & $\left[ s_{0}x_{2},s_{2}y_{2}\right] \left[ s_{2}y_{2},s_{1}x_{2}%
\right] $ \\ \hline
&  &  & $\left[ s_{1}x_{2},s_{1}y_{2}\right] \left[ s_{1}y_{2},s_{0}x_{2}%
\right] $ \\ \hline $7)$ & $d_{4}(F_{(1,0)(3,2)}(x_{2},y_{2}))$ &
$=$ & $\left[
s_{1}s_{0}d_{2}x_{2},s_{2}y_{2}\right] \left[ s_{2}y_{2},s_{2}s_{0}d_{2}x_{2}%
\right] \left[ s_{0}x_{2},s_{2}y_{2}\right] $ \\ \hline $8)$ &
$d_{4}(F_{(1)(3,2)}(x_{3},x_{2}))$ & $=$ & $\left[
s_{1}d_{3}x_{3},s_{2}x_{2}\right] \left[ s_{2}x_{2},s_{2}d_{3}x_{3}\right] %
\left[ x_{3},s_{2}x_{2}\right] $ \\ \hline $9)$ &
$d_{4}(F_{(0)(3,2)}(x_{3},x_{2}))$ & $=$ & $\left[
s_{0}d_{3}x_{3},s_{2}x_{2}\right] $ \\ \hline $10)$ &
$d_{4}(F_{(3,1)(2)}(x_{3},x_{2}))$ & $=$ & $\left[
s_{0}d_{3}x_{3},s_{1}x_{2}\right] \left[
s_{1}x_{2},s_{1}d_{3}x_{3}\right] $
\\ \hline
&  &  & $\left[ s_{2}d_{3}x_{3},s_{2}x_{2}\right] \left[ s_{2}x_{2},x_{3}%
\right] $ \\ \hline $11)$ & $d_{4}(F_{(0)(2,1)}(x_{3},x_{2}))$ & $=$
& $\left[ s_{0}d_{3}x_{3},s_{2}s_{1}d_{2}x_{2}\right] \left[
s_{2}s_{1}d_{2}x_{2},s_{1}d_{3}x_{3}\right] $ \\ \hline &  &  &
$\left[ s_{2}d_{3}x_{3},s_{2}s_{1}d_{2}x_{2}\right] \left[
s_{1}x_{2},x_{3}\right] $ \\ \hline $12)$ &
$d_{4}(F_{(3,1)(2)}(x_{2},x_{3}))$ & $=$ & $\left[
s_{1}x_{2},s_{2}d_{3}x_{3}\right] \left[
s_{2}d_{3}x_{3},s_{2}x_{2}\right] $
\\ \hline
&  &  & $\left[ s_{2}l,x_{3}\right] \left[ x_{3},s_{1}x_{2}\right] $ \\
\hline $13)$ & $d_{4}(F_{(2,1)(3)}(x_{2},x_{3}))$ & $=$ & $\left[
s_{2}s_{1}d_{2}x_{2},x_{3}\right] \left[ x_{3},s_{1}x_{2}\right] $
\\ \hline $14)$ & $d_{4}(F_{(3,0)(2)}(x_{2},x_{3}))$ & $=$ & $\left[
s_{0}x_{2},s_{2}d_{3}x_{3}\right] \left[ x_{3},s_{0}x_{2}\right] $
\\ \hline $15)$ & $d_{4}(F_{(3,0)(1)}(x_{2},x_{3}))$ & $=$ & $\left[
s_{0}x_{2},s_{1}d_{3}x_{3}\right] \left[
s_{1}d_{3}x_{3},s_{1}x_{2}\right] $
\\ \hline
&  &  & $\left[ s_{2}x_{2},s_{2}d_{3}x_{3}\right] \left[ x_{3},s_{2}x_{2}%
\right] $ \\ \hline $16)$ & $d_{4}(F_{(2,0)(3)}(x_{2},x_{3}))$ & $=$
& $\left[ s_{2}s_{0}d_{2}x_{2},x_{3}\right] \left[
x_{3},s_{0}x_{2}\right] $ \\ \hline $17)$ &
$d_{4}(F_{(2,0)(1)}(x_{2},x_{3}))$ & $=$ & $\left[
s_{2}s_{0}d_{2}x_{2},s_{1}d_{3}x_{3}\right] \left[
s_{1}d_{3}x_{3},s_{2}s_{1}d_{2}x_{2}\right] $ \\ \hline &  &  &
$\left[ s_{2}s_{1}d_{2}x_{2},s_{2}d_{3}x_{3}\right] \left[
s_{2}d_{3}x_{3},s_{2}s_{0}d_{2}x_{2}\right] $ \\ \hline
&  &  & $\left[ s_{0}x_{2},x_{3}\right] \left[ x_{3},s_{1}x_{2}\right] $ \\
\hline $18)$ & $d_{4}(F_{(1,0)(3)}(x_{2},x_{3}))$ & $=$ & $\left[
s_{1}s_{0}d_{2}x_{2},x_{3}\right] $ \\ \hline $19)$ &
$d_{4}(F_{(1,0)(2)}(x_{2},x_{3}))$ & $=$ & $\left[
s_{1}s_{0}d_{2}x_{2},s_{2}d_{3}x_{3}\right] \left[
s_{2}d_{3}x_{3},s_{2}s_{0}d_{2}x_{2}\right] $ \\ \hline &  &  &
$\left[ s_{0}x_{2},x_{3}\right] $ \\ \hline $20)$ &
$d_{4}(F_{(2)(3)}(x_{3},y_{3}))$ & $=$ & $\left[
s_{2}d_{3}x_{3},y_{3}\right] \left[ y_{3},x_{3}\right] $ \\ \hline
$21)$ & $d_{4}(F_{(1)(3)}(x_{3},y_{3}))$ & $=$ & $\left[
s_{1}d_{3}x_{3},y_{3}\right] $ \\ \hline $22)$ &
$d_{4}(F_{(0)(3)}(x_{3},y_{3}))$ & $=$ & $\left[
s_{0}d_{3}x_{3},y_{3}\right] $ \\ \hline $23)$ &
$d_{4}(F_{(1)(2)}(x_{3},y_{3}))$ & $=$ & $\left[
s_{1}d_{3}x_{3},s_{2}d_{3}y_{3}\right] \left[ s_{2}d_{3}y_{3},s_{2}d_{3}x_{3}%
\right] \left[ x_{3},y_{3}\right] $ \\ \hline $24)$ &
$d_{4}(F_{(0)(2)}(x_{3},y_{3}))$ & $=$ & $\left[
s_{0}d_{3}x_{3},s_{2}d_{3}y_{3}\right] $ \\ \hline $25)$ &
$d_{4}(F_{(0)(1)}(x_{3},y_{3}))$ & $=$ & $\left[
s_{0}d_{3}x_{3},s_{1}d_{3}y_{3}\right] \left[ s_{1}d_{3}y_{3},s_{1}d_{3}x_{3}%
\right] $ \\ \hline
&  &  & $\left[ s_{2}d_{3}x_{3},s_{2}d_{3}y_{3}\right] \left[ y_{3},x_{3}%
\right] $ \\ \hline
\end{tabular}

\begin{center}
Table 1
\end{center}

where $x_{3},y_{3}\in NG_{3},x_{2},y_{2}\in NG_{2},x_{1}\in NG_{1}$.

\newpage

\section{$2$-Crossed Modules}

The notion of crossed module is an efficient algebraic tool to
handle connected spaces with only the first homotopy groups non
trivial, up to homotopy.

A $crossed$ $module$ is a group homomorphism $\partial :M\rightarrow
P$
together with an action of $P$ on $M$, written $^{p}m$ for $p\in P$ and $%
m\in M$, satisfying the conditions.

\textbf{CM1)} $\partial $ is $P$-equivariant, i.e, for all $p\in P$
, $m\in
M $%
\begin{equation*}
\partial (^{p}m)=p\partial (m)p^{-1}
\end{equation*}

\textbf{CM2) }(Peiffer Identity) for all $m,m^{\prime }\in M$%
\begin{equation*}
^{\partial m}m^{\prime }=mm^{\prime }m^{-1}
\end{equation*}%
We will denote such a crossed module by $(M,P,\partial )$.

A \textit{morphism of crossed module} from $(M,P,\partial )$ to
$(M^{\prime
},P^{\prime },\partial ^{\prime })$ is a pair of group homomorphisms%
\begin{equation*}
\phi :M\longrightarrow M^{\prime }\text{ , \ }\psi :P\longrightarrow
P^{\prime }\text{ }
\end{equation*}%
such that $\phi (^{p}m)=^{\psi (p)}\phi (m)$ and $\partial ^{\prime
}\phi (m)=\psi \partial (m)$.

We thus get a category $\mathbf{XMod}$ of crossed modules.

\textit{Examples of Crossed Modules}

1) Any normal subgroup $N\trianglelefteq P$ gives rise to a crossed
module namely the inclusion map, $i:N\hookrightarrow P$. Conversely,
given any crossed module $\partial :M\longrightarrow P$ ,
$Im\partial $ is a normal subgroup of $P$.

2) Given any $P$-module $M$, the trivial map%
\begin{equation*}
1:M\longrightarrow P
\end{equation*}%
that maps everything to $1$ in $P$, is a crossed module. Conversely,
if $\partial :M\rightarrow P$ is a crossed module, $\ker \partial $
is central in $M$ and inherits a natural $P$-module structure from the $P$%
-action on $M$.

The following definition of $2$-crossed module is equivalent to that
given by Conduch{\'{e}}, \cite{conduche}.

A 2-crossed module of groups consists of a complex of groups%
\begin{equation*}
L\overset{\partial _{2}}{\longrightarrow }M\overset{\partial _{1}}{%
\longrightarrow }N
\end{equation*}%
together with (a) actions of $N$ on $M$ and $L$ so that $\partial
_{2},\partial _{1}$ are morphisms of $N$-groups, and (b) an
$N$-equivariant function
\begin{equation*}
\{\quad ,\quad \}:M\times M\longrightarrow L
\end{equation*}%
called a \textit{Peiffer lifting}. This data must satisfy the
following axioms:

\newpage

\begin{equation*}
\begin{array}{lrrll}
\mathbf{2CM1)} &  & \partial _{2}\{m,m^{\prime }\} & = & \left(
^{\partial _{1}m}m^{\prime }\right) mm^{\prime
}{}^{-1}m^{-1}\newline
\\
\mathbf{2CM2)} &  & \{\partial _{2}l,\partial _{2}l^{\prime }\} & =
& [l^{\prime },l]\newline
\\
\mathbf{2CM3)} &  & (i)\quad \{mm^{\prime },m^{\prime \prime }\} & =
& ^{\partial _{1}m}\{m^{\prime },m^{\prime \prime }\}\{m,m^{\prime
}m^{\prime \prime }m^{\prime }{}^{-1}\}\newline
\\
&  & (ii)\quad \{m,m^{\prime }m^{\prime \prime }\} & = &
\{m,m^{\prime }\}^{mm^{\prime }m^{-1}}\{m,m^{\prime \prime
}\}\newline
\\
\mathbf{2CM4)} &  & \{m,\partial _{2}l\}\{\partial _{2}l,m\} & = &
^{\partial _{1}m}ll^{-1}\newline
\\
\mathbf{2CM5)} &  & ^{n}\{m,m^{\prime }\} & = & \{^{n}m,^{n}m^{\prime }\}%
\newline
\end{array}%
\end{equation*}%
\newline
for all $l,l^{\prime }\in L$, $m,m^{\prime },m^{\prime \prime }\in M$ and $%
n\in N$.

Here we have used $^{m}l$ as a shorthand for $\{\partial _{2}l,m\}l$
in
condition $\mathbf{2CM3)}(ii)$ where $l$ is $\{m,m^{\prime \prime }\}$ and $%
m $ is $mm^{\prime }(m)^{-1}$. This gives a new action of $M$ on
$L$. Using this notation, we can split $\mathbf{2CM4)}$ into two
pieces, the first of which is tautologous:
\begin{equation*}
\begin{array}{lrrll}
\mathbf{2CM4)} & \quad (a)\quad \{\partial _{2}l,m\} & = &
^{m}l(l)^{-1}, &
\\
& \quad (b)\quad \{m,\partial _{2}l\} & = & (^{\partial
_{1}m}l)(^{m}l^{-1}). &
\end{array}%
\end{equation*}%
The old action of $M$ on $L$, via $\partial _{1}$ and the $N$-action
on $L$, is in general distinct from this second action with
$\{m,\partial _{2}l\}$
measuring the difference (by $\mathbf{2CM4)}(b)$). An easy argument using $%
\mathbf{2CM2)}$ and $\mathbf{2CM4)}(b)$ shows that with this action, $^{m}l$%
, of $M$ on $L$, $(L,M,\partial _{2})$ becomes a crossed module. A
morphism of $2$-crossed modules can be defined in an obvious way. We
thus define the category of $2$-crossed modules denoting it by
$\mathbf{X}_{2}\mathbf{Mod}$.

A crossed square as defined by D. Guin-Wal\'{e}ry and J.-L. Loday in
\cite{wl} (see
also \cite{jl}, \cite{bl1}), can be seen as a mapping cone in \cite%
{cond}. Furthermore 2-crossed modules are related to simplicial
groups. This relation can be found in \cite{conduche}, \cite{amut4}.

\begin{theorem}
The category $\mathbf{X_{2}Mod}$ of $2$-crossed modules is
equivalent to the category of $\mathbf{SimpGrp}_{\leq2}$ simplicial
groups with Moore complex of length $2$.
\end{theorem}

\section{$3$-Crossed Modules}

In the following we will define the category of $3$-crossed modules.
First of all we adapt ideas from Conduch{\'{e}}'s method given in
\cite{conduche}. He gave some equalities by using the semi-direct
decomposition of a
simplicial group. But these are exactly the images of Peiffer pairings $%
F_{\alpha ,\beta }$ under $\partial _{3}$ for $n=3$ case defined in \cite%
{amut4}. The difference of our method is to use $F_{\alpha ,\beta }$
instead
of semi-direct decomposition. Thus we will define similar equalities for $%
n=4 $ and get the axioms of 3-crossed module.

Let $\mathbf{G}$ be a simplicial group with Moore complex of length $3$ and $%
NG_{0}=N,$ $NG_{1}=M,$ $NG_{2}=L,$ $NG_{3}=K$. Thus we have a group
complex
\begin{equation*}
K\overset{\partial _{3}}{\longrightarrow }L\overset{\partial _{2}}{%
\longrightarrow }M\overset{\partial _{1}}{\longrightarrow }N
\end{equation*}%
Let the actions of $N$ on $K$, $L$, $M$, $M$ on $L$, $K$ and $L$ on
$K$ be
as follows;%
\begin{equation}
\begin{array}{ll}
^{n}m & =s_{0}n\left( m\right) s_{0}n^{-1} \\
^{n}l & =s_{1}s_{0}n\left( l\right) s_{1}s_{0}n^{-1} \\
^{n}k & =s_{2}s_{1}s_{0}n\left( k\right) s_{2}s_{1}s_{0}n^{-1} \\
^{m}l & =s_{1}m\left( l\right) s_{1}m^{-1} \\
^{m}k & =s_{2}s_{1}m\left( k\right) s_{2}s_{1}m^{-1} \\
l\cdot k & =s_{2}l\left( k\right) s_{2}l^{-1}.%
\end{array}%
\end{equation}

\bigskip

Using table 1, since%
\begin{equation*}
\begin{array}{rr}
\left[ s_{1}s_{0}ms_{2}s_{1}\partial _{1}m,k\right] & =1 \\
\left[ s_{1}ls_{2}s_{1}\partial _{2}l,k\right] & =1 \\
\left[ k^{\prime },k^{-1}s_{2}\partial _{3}k\right] & =1%
\end{array}%
\end{equation*}%
we get

\begin{equation*}
\begin{array}[t]{ll}
^{\partial _{1}m}k & =s_{1}s_{0}m\left( k\right) s_{1}s_{0}m^{-1} \\
^{\partial _{2}l}k & =s_{1}l\left( k\right) s_{1}l^{-1} \\
\partial _{3}k\cdot k^{\prime } & =k\left( k^{\prime }\right) k^{-1}%
\end{array}%
\end{equation*}%
and using the simplicial identities we get,
\begin{equation*}
\partial _{3}(l\cdot k)=\partial _{3}(s_{2}l\left( k\right)
s_{2}l^{-1})=\partial _{3}s_{2}l\left( \partial _{3}k\right)
s_{2}l^{-1}=l\left( \partial _{3}k\right) l^{-1}.
\end{equation*}%
Thus $\partial _{3}:K\rightarrow L$ is a crossed module.

The Peiffer liftings given in the definition below are the
$F_{\alpha ,\beta }$ pairings for the case $n=3$ defined in
\cite{c}.

\begin{definition}
Let $K\overset{\partial _{3}}{\longrightarrow }L\overset{\partial _{2}}{%
\longrightarrow }M\overset{\partial _{1}}{\longrightarrow }N$ be a
group complex defined above. We define the Peiffer liftings as
follows;
\begin{equation*}
\begin{array}{llll}
\left\{ ~,~\right\} : & M\times M & \mathbf{\longrightarrow } & L \\
& \left\{ m,m^{\prime }\right\} & = & \left[ s_{1}m,s_{1}m^{\prime }\right] %
\left[ s_{1}m^{\prime },s_{0}m\right] \\
\left\{ ~,~\right\} _{(1)(0)}: & L\times L & \longrightarrow & K \\
& \left\{ l,l^{\prime }\right\} _{_{(1)(0)}} & = & \left[
s_{2}l^{\prime },s_{2}l\right] \left[ s_{1}l,s_{1}l^{\prime }\right]
\left[ s_{1}l^{\prime
},s_{0}l\right] \\
\left\{ ~,~\right\} _{(2)(1)}: & L\times L & \longrightarrow & K \\
& \left\{ l,l^{\prime }\right\} _{(2)(1)} & = & \left[
s_{2}l,s_{2}l^{\prime
}\right] \left[ s_{2}l^{\prime },s_{1}l\right] \\
\left\{ ~,~\right\} _{(0)(2)}: & L\times L & \longrightarrow & K \\
& \left\{ l,l^{\prime }\right\} _{_{(0)(2)}} & = & \left[
s_{2}l^{\prime
},s_{0}l\right] \\
\left\{ ~,~\right\} _{(1,0)(2)} & M\times L & \longrightarrow & K \\
& \left\{ m,l^{\prime }\right\} _{_{(1,0)(2)}} & = & \left[
s_{2}s_{0}m,s_{2}l^{\prime }\right] \left[ s_{2}l^{\prime },s_{1}s_{0}m%
\right] \\
\left\{ ~,~\right\} _{(2,0)(1)} & M\times L & \longrightarrow & K \\
& \left\{ m,l^{\prime }\right\} _{(2,0)(1)} & = & \left[
s_{2}s_{0}m,s_{2}l^{\prime }\right] \left[ s_{2}l^{\prime },s_{2}s_{1}m%
\right] \left[ s_{2}s_{1}m,s_{1}l^{\prime }\right] \left[
s_{1}l^{\prime
},s_{2}s_{0}m\right] \\
\left\{ ~,~\right\} _{(0)(2,1)} & L\times M & \longrightarrow & K \\
& \left\{ l^{\prime },m\right\} _{(0)(2,1)} & = & \left[
s_{2}s_{1}m,s_{2}l^{\prime }\right] \left[ s_{1}l^{\prime },s_{2}s_{1}m%
\right] \left[ s_{2}s_{1}m,s_{0}l^{\prime }\right]%
\end{array}%
\end{equation*}%
where \ $m,$ $m^{\prime }\in M,$ $l,$ $l^{\prime }\in L$.
\end{definition}

Then using Table 1 we get the following identities.

\begin{center}
\begin{tabular}{|l|l|l|}
\hline $\left\{ m,\partial _{3}y_{3}\right\} _{(1,0)(2)}$ & $=$ &
$\left\{ m,\partial _{3}y_{3}\right\} _{(2,0)(1)}\text{
}^{m}(y_{3})^{\partial _{1}m}(y_{3}^{-1})$ \\ \hline
$\left\{ \partial _{3}y_{3},m\right\} _{(0)(2,1)}$ & $=$ & $%
^{m}(y_{3})y_{3}^{-1}$ \\ \hline $\left\{ m,\partial _{3}k\right\}
_{(1,0)(2)}$ & $=$ & $\left\{ m,\partial
_{3}k\right\} _{(2,0)(1)}\text{ }\left\{ \partial _{3}k,m\right\} _{(0)(2,1)}%
\text{ }k$ $^{\partial _{1}m}(k^{-1})$ \\ \hline $\left\{ l^{\prime
},\partial _{2}l\right\} _{(0)(2,1)}$ & $=$ & $\left\{ l,l^{\prime
}\right\} _{(2)(1)}^{-1}\left\{ l^{\prime },l\right\} _{(1)(0)}$
\\ \hline
$\left\{ \partial _{2}l,l^{\prime }\right\} _{(2,0)(1)}$ & $=$ &
$\left\{ l,l^{\prime }\right\} _{(0)(2)\ }^{-1}$ $^{\left[ l^{\prime
},\ l\right] }(\left\{ l,l^{\prime }\right\} _{(2)(1)})\left\{
l,l^{\prime }\right\} _{(1)(0)}$ \\ \hline $\left\{ \partial
_{2}l,l^{\prime }\right\} _{(1,0)(2)}$ & $=$ & $(\left\{ l,l^{\prime
}\right\} _{(0)(2)})^{-1}$ \\ \hline
$\left\{ l,l^{\prime }l^{\prime \prime }\right\} _{(2)(1)}$ & $=$ & $%
\{l,l^{\prime }\}_{(2)(1)}{}^{\partial l}l^{\prime }.\{l,l^{\prime
\prime }\}_{(2)(1)}$ \\ \hline
$\left\{ ll^{\prime },l^{\prime \prime }\right\} _{(2)(1)}$ & $=$ & $%
l.\{l^{\prime },l^{\prime \prime }\}_{(2)(1)}\{l,^{\partial
l^{\prime }}l^{\prime \prime }\}_{(2)(1)}$ \\ \hline $\partial
_{3}(\left\{ l,l^{\prime }\right\} _{(1)(0)})$ & $=$ & $\left[
l,l^{\prime }\right] \left\{ \partial _{2}l,\partial _{2}l^{\prime
}\right\} $ \\ \hline $\partial _{3}(\left\{ l,l^{\prime }\right\}
_{(2)(1)})$ & $=$ & $ll^{\prime }l^{-1}(^{\partial _{2}l}l^{\prime
})^{-1}$ \\ \hline $\partial _{3}(\left\{ l,l^{\prime }\right\}
_{(0)(2)})$ & $=$ & $\partial _{3}(\left\{ \partial _{2}l,l^{\prime
}\right\} _{(1,0)(2)})^{-1}$ \\ \hline $\partial _{3}\left\{
l,m\right\} _{(0)(2,1)}$ & $=$ & $^{m}ll^{-1}\left\{
\partial _{2}l,m\right\} $ \\ \hline
$\partial _{3}\left\{ m,l\right\} _{(2,0)(1)}$ & $=$ & $\partial
_{3}\left\{ m,l\right\} _{(1,0)(2)}\text{ }^{\partial _{1}m}l\text{
}^{m}(l^{-1})\left\{ m,\partial _{2}l\right\} $ \\ \hline $\left\{
\partial _{3}k,l\right\} _{(2)(1)}\left\{ l,\partial _{3}k\right\}
_{(2)(1)}$ & $=$ & $k\left( ^{\partial _{2}l}(k^{-1})\right) $ \\
\hline $\left\{ \partial _{3}k,l\right\} _{(1)(0)}\left\{ l,\partial
_{3}k\right\} _{(1)(0)}$ & $=$ & $1$ \\ \hline
$\left\{ \partial _{3}k,\partial _{3}k^{\prime }\right\} _{(2)(1)}$ & $=$ & $%
\left[ k,k^{\prime }\right] $ \\ \hline
$\left\{ \partial _{3}k,\partial _{3}k^{\prime }\right\} _{(1)(0)}$ & $=$ & $%
\left[ k^{\prime },k\right] $ \\ \hline $\left\{ \partial
_{3}k,l^{\prime }\right\} _{(0)(2)}$ & $=$ & $1$ \\ \hline $\left\{
\partial _{2}l,\partial _{3}k\right\} _{(1,0)(2)}$ & $=$ & $\left\{
l,\partial _{3}k\right\} _{(0)(2)}^{-1}$ \\ \hline $\left\{ \partial
_{2}l,\partial _{3}k\right\} _{(2,0)(1)}$ & $=$ & $\left\{
l,\partial _{3}k\right\} _{(0)(2)}k\left( ^{\partial
_{2}l}(k^{-1})\right) $
\\ \hline
$\left\{ \partial _{3}k,\partial _{2}l\right\} _{(0)(2,1)}$ & $=$ & $%
^{\partial _{2}l}k\text{ }k^{-1}$ \\ \hline
\end{tabular}

Table 2

\begin{tabular}{|l|}
\hline $^{n}\left\{ m,m^{\prime }\right\} =\left\{
^{n}m,^{n}m^{\prime }\right\} $
\\ \hline
$^{n}\left\{ l,l^{\prime }\right\} _{_{(1)(0)}}=\left\{
^{n}l,^{n}l^{\prime }\right\} _{_{(1)(0)}}$ \\ \hline $^{n}\left\{
l,l^{\prime }\right\} _{(2)(1)}=\left\{ ^{n}l,^{n}l^{\prime
}\right\} _{(2)(1)}$ \\ \hline $^{n}\left\{ l,l^{\prime }\right\}
_{_{(0)(2)}}=\left\{ ^{n}l,^{n}l^{\prime }\right\} _{_{(0)(2)}}$ \\
\hline $^{n}\left\{ m,l^{\prime }\right\} _{_{(1,0)(2)}}=\left\{
^{n}m,^{n}l^{\prime }\right\} _{_{(1,0)(2)}}$ \\ \hline $^{n}\left\{
m,l^{\prime }\right\} _{(2,0)(1)}=\left\{ ^{n}m,^{n}l^{\prime
}\right\} _{(2,0)(1)}$ \\ \hline $^{n}\left\{ l^{\prime },m\right\}
_{(0)(2,1)}=\left\{ ^{n}l^{\prime },^{n}m\right\} _{(0)(2,1)}$ \\
\hline
\end{tabular}

Table 3

\begin{tabular}{|l|}
\hline $^{m}\left\{ m^{\prime },m^{\prime \prime }\right\} =$
$^{m}\left\{ m^{\prime },^{m}m^{\prime \prime }\right\} $ \\ \hline
$^{m}\left\{ l,l^{\prime }\right\} _{_{(1)(0)}}=\left\{
^{m}l,^{m}l^{\prime }\right\} _{_{(1)(0)}}$ \\ \hline $^{m}\left\{
l,l^{\prime }\right\} _{(2)(1)}=\left\{ ^{m}l,^{m}l^{\prime
}\right\} _{(2)(1)}$ \\ \hline $^{m}\left\{ l,l^{\prime }\right\}
_{_{(0)(2)}}=\left\{ ^{m}l,^{m}l^{\prime }\right\} _{_{(0)(2)}}$ \\
\hline $^{m}\left\{ m,l^{\prime }\right\} _{_{(1,0)(2)}}=\left\{
^{m}m,^{m}l^{\prime }\right\} _{_{(1,0)(2)}}$ \\ \hline $^{m}\left\{
m,l^{\prime }\right\} _{(2,0)(1)}=\left\{ ^{m}m,^{m}l^{\prime
}\right\} _{(2,0)(1)}$ \\ \hline $^{m}\left\{ l^{\prime },m\right\}
_{(0)(2,1)}=\left\{ ^{m}l^{\prime },^{m}m\right\} _{(0)(2,1)}$ \\
\hline
\end{tabular}

Table 4
\end{center}

\noindent where $m,$ $m^{\prime },m^{\prime \prime }\in M,$ $l,$
$l^{\prime }\in
L,k,k^{\prime }\in K$. From these results all liftings are $N,$ $M$%
-equivariant.

\begin{definition}
A \textit{3-crossed module} consists of a complex of groups
\begin{equation*}
K\overset{\partial _{3}}{\longrightarrow }L\overset{\partial _{2}}{%
\longrightarrow }M\overset{\partial _{1}}{\longrightarrow }N
\end{equation*}%
together with an action of $N$ on $K,L,M$ and an action of $M$ on
$K,L$ and an action of $L$ on $K$ so that $\partial _{3}$, $\partial
_{2}$,$\partial _{1}$ are morphisms of $N,M$-groups and the
$M,N$-equivariant liftings
\begin{equation*}
\begin{tabular}{lll}
$\{$ $,$ $\}_{(1)(0)}:L\times L\longrightarrow K,$ & $\{$ $,$ $%
\}_{(0)(2)}:L\times L\longrightarrow K,$ & $\{$ $,$
$\}_{(2)(1)}:L\times
L\longrightarrow K,$ \\
&  &  \\
$\{$ $,$ $\}_{(1,0)(2)}:M\times L\longrightarrow K,$ & $\{$ $,$ $%
\}_{(2,0)(1)}:M\times L\longrightarrow K,$ &  \\
&  &  \\
$\{$ $,$ $\}_{(0)(2,1)}:L\times M\longrightarrow K,$ & $\{$ $,$
$\}:M\times M\longrightarrow L$ &
\end{tabular}%
\end{equation*}%
called \textit{$3$-dimensional Peiffer liftings}. This data must
satisfy the following axioms:

\begin{equation*}
\begin{array}{lrrl}
\mathbf{3CM1)} &  & K\overset{\partial _{3}}{\longrightarrow }L\overset{%
\partial _{2}}{\longrightarrow }M & \text{is a }2\text{-crossed module with
the Peiffer lifting }\{\text{ },\text{ }\}_{(2,1)} \\
\mathbf{3CM2)} &  & \left\{ m,\partial _{3}k\right\} _{(1,0)(2)} &
=\left\{ m,\partial _{3}k\right\} _{(2,0)(1)}\text{
}^{m}(k)^{\partial _{1}m}(k^{-1})
\\
\mathbf{3CM3)} &  & \left\{ \partial _{3}k,m\right\} _{(0)(2,1)} & =\text{ }%
^{m}(k)k^{-1} \\
\mathbf{3CM4)} &  & \left\{ m,\partial _{3}k\right\} _{(1,0)(2)} &
=\left\{ m,\partial _{3}k\right\} _{(2,0)(1)}\text{ }\left\{
\partial _{3}k,m\right\} _{(0)(2,1)}\text{ }k^{\partial
_{1}m}(k^{-1})\newline
\\
\mathbf{3CM5)} &  & \left\{ l^{\prime },\partial _{2}l\right\}
_{(0)(2,1)} & =\left\{ l,l^{\prime }\right\} _{(2)(1)}^{-1}\left\{
l^{\prime },l\right\} _{(1)(0)}\newline
\\
\mathbf{3CM6)} &  & \left\{ \partial _{2}l,l^{\prime }\right\}
_{(2,0)(1)} & =\left\{ l,l^{\prime }\right\} _{(0)(2)\ }^{-1}\text{
}^{\left[ l^{\prime },\ l\right] }(\left\{ l,l^{\prime }\right\}
_{(2)(1)})\left\{ l,l^{\prime
}\right\} _{(1)(0)} \\
\mathbf{3CM7)} &  & \left\{ \partial _{2}l,l^{\prime }\right\}
_{(1,0)(2)} &
=(\left\{ l,l^{\prime }\right\} _{(0)(2)})^{-1} \\
\mathbf{3CM8)} &  & \partial _{3}(\left\{ l,l^{\prime }\right\}
_{(1)(0)}) & =\left[ l,l^{\prime }\right] \left\{ \partial
_{2}l,\partial _{2}l^{\prime
}\right\}  \\
\mathbf{3CM9)} &  & \partial _{3}(\left\{ l,l^{\prime }\right\}
_{(0)(2)}) & =\partial _{3}(\left\{ \partial _{2}l,l^{\prime
}\right\} _{(1,0)(2)})^{-1}
\\
\mathbf{3CM10)} &  & \partial _{3}\left\{ l,m\right\} _{(0)(2,1)} & =\text{ }%
^{m}ll^{-1}\left\{ \partial _{2}l,m\right\}  \\
\mathbf{3CM11)} &  & \partial _{3}\left\{ m,l\right\} _{(2,0)(1)} &
=\partial _{3}\left\{ m,l\right\} _{(1,0)(2)}\text{ }^{\partial
_{1}m}l\text{
}^{m}(l^{-1})\left\{ m,\partial _{2}l\right\}  \\
\mathbf{3CM12a)} &  & \left\{ \partial _{3}k,l\right\} _{(1)(0)} &
=(^{l}k)k^{-1} \\
\mathbf{3CM12b)} &  & \left\{ l,\partial _{3}k\right\} _{(1)(0)} &
k(^{l}k)^{-1} \\
\mathbf{3CM13)} &  & \left\{ \partial _{3}k,\partial _{3}k^{\prime
}\right\}
_{(1)(0)} & =\left[ k^{\prime },k\right]  \\
\mathbf{3CM14)} &  & \left\{ \partial _{3}k,l^{\prime }\right\}
_{(0)(2)} &
=1 \\
\mathbf{3CM15)} &  & \left\{ \partial _{2}l,\partial _{3}k\right\}
_{(1,0)(2)} & =\left\{ l,\partial _{3}k\right\} _{(0)(2)}^{-1} \\
\mathbf{3CM16)} &  & \left\{ \partial _{2}l,\partial _{3}k\right\}
_{(2,0)(1)} & =\left\{ l,\partial _{3}k\right\} _{(0)(2)}k\left(
^{\partial
_{2}l}(k^{-1})\right)  \\
\mathbf{3CM17)} &  & \left\{ \partial _{3}k,\partial _{2}l\right\}
_{(0)(2,1)} & =\text{ }^{\partial _{2}l}k\text{ }k^{-1} \\
\mathbf{3CM18)} &  & \partial _{2}\left\{ m,m^{\prime }\right\}  &
=mm^{\prime }m^{-1}(^{\partial _{1}m}m^{\prime })^{-1}%
\end{array}%
\end{equation*}
\end{definition}

\newpage

We denote such a 3-crossed module by $(K,L,M,N,\partial
_{3},\partial _{2},\partial _{1}).$

A \textit{morphism of }$3$\textit{-crossed modules} of groups may be
pictured by the diagram

\begin{center}
$ \xymatrix@R=40pt@C=40pt{
  L_{3}
            \ar[d]_-{f_{3}}
             \ar@{->}@<0pt>[r]^-{\partial_{3}}
 &
 L_{2}
                     \ar[d]_-{f_{2}}
                     \ar@{->}@<0pt>[r]^-{\partial_{2}}
 &
 L_{1}
    \ar[d]_-{f_{1}}
    \ar@{->}@<0pt>[r]^-{\partial_{1}}
 &
  L_{0}
        \ar[d]_-{f_{0}}
  \\
   L'_{3}
             \ar@{->}@<0pt>[r]_-{\partial'_{3}}
 &
  L'_{2}
    \ar@{->}@<0pt>[r]_-{\partial'_{2}}
 &
  L'_{1}
   \ar@{->}@<0pt>[r]_-{\partial'_{1}}
 &
 L'_{0}
    }
 $

\end{center}

\noindent where
\begin{equation*}
f_{1}(^{n}m)=\text{ }^{(f_{0}(n))}f_{1}(m),\text{ }f_{2}(^{n}l)=\text{ }%
^{(f_{0}(n))}f_{2}(l),\text{ }f_{3}(^{n}k)=\text{
}^{(f_{0}(n))}f_{3}(k)
\end{equation*}%
for $\left\{ \text{ },\text{ }\right\} _{(0)(2)},\left\{ \text{ },\text{ }%
\right\} _{(2)(1)},$ $\left\{ \text{ },\text{ }\right\} _{(1)(0)}$
\begin{equation*}
\left\{ \text{ },\text{ }\right\} f_{2}\times f_{2}=f_{3}\left\{ \text{ },%
\text{ }\right\}
\end{equation*}%
for $\left\{ \text{ },\text{ }\right\} _{(1,0)(2)},\left\{ \text{ },\text{ }%
\right\} _{(2,0)(1)}$
\begin{equation*}
\left\{ \text{ },\text{ }\right\} f_{1}\times f_{2}=f_{3}\left\{ \text{ },%
\text{ }\right\}
\end{equation*}%
for $\left\{ \text{ },\text{ }\right\} _{(0)(2,1)}$
\begin{equation*}
\left\{ \text{ },\text{ }\right\} f_{2}\times f_{1}=f_{3}\left\{ \text{ },%
\text{ }\right\}
\end{equation*}%
and for $\left\{ \text{ },\text{ }\right\} $
\begin{equation*}
\left\{ \text{ },\text{ }\right\} f_{1}\times f_{1}=f_{2}\left\{ \text{ },%
\text{ }\right\}
\end{equation*}%
for all $k\in K,l\in L,m\in M,n\in N$. These compose in an obvious
way. We
thus can define the category of $3$-crossed modules, denoting it by $\mathbf{%
X}_{3}\mathbf{Mod}$.

\section{Applications}

\subsection{Simplicial Groups}

As an application we consider in details the relation between
simplicial groups and $3$-crossed modules.

\begin{proposition}
Let $\mathbf{G}$ be a simplicial group with Moore complex
$\mathbf{NG}$. Then the group complex
\begin{equation*}
NG_{3}/\partial _{4}(NG_{4}\cap D_{4})\overset{\overline{\partial }_{3}}{%
\longrightarrow }NG_{2}\overset{\partial _{2}}{\longrightarrow }NG_{1}%
\overset{\partial _{1}}{\longrightarrow }NG_{0}
\end{equation*}%
is a $3$-crossed module with the Peiffer liftings defined below:
\begin{equation*}
\begin{array}{rr}
\{\text{ },\text{ }\}:NG_{1}\times NG_{1} & \longrightarrow \\
\left\{ x_{1},y_{1}\right\} & \longmapsto \\
&  \\
\{\text{ },\text{ }\}_{(1)(0)}:NG_{2}\times NG_{2} & \longrightarrow \\
\left\{ x_{2},y_{2}\right\} & \longmapsto \\
&  \\
\{\text{ },\text{ }\}_{(2)(1)}:NG_{2}\times NG_{2} & \longrightarrow \\
\left\{ x_{2},y_{2}\right\} & \longmapsto \\
&  \\
\{\text{ },\text{ }\}_{(0)(2)}:NG_{2}\times NG_{2} & \longrightarrow \\
\left\{ x_{2},y_{2}\right\} & \longmapsto \\
&  \\
\{\text{ },\text{ }\}_{(1,0)(2)}:NG_{1}\times NG_{2} & \longrightarrow \\
\left\{ x_{1},y_{2}\right\} & \longmapsto \\
&  \\
\{\text{ },\text{ }\}_{(2,0)(1)}:NG_{1}\times NG_{2} & \longrightarrow \\
\left\{ x_{1},y_{2}\right\} & \longmapsto \\
&  \\
\{\text{ },\text{ }\}_{(0)(2,1)}:NG_{2}\times NG_{1} & \longrightarrow \\
\left\{ y_{2},x_{1}\right\} & \longmapsto%
\end{array}%
\begin{array}{l}
NG_{2} \\
\left[ s_{0}x_{1},s_{1}y_{1}\right] \left[ s_{1}y_{1},s_{1}x_{1}\right] \\
\\
NG_{3}/\partial _{4}(NG_{4}\cap D_{4}) \\
\overline{\left( \left[ s_{0}x_{2},s_{1}y_{2}\right] \left[
s_{1}y_{2},s_{1}x_{2}\right] \left[ s_{2}x_{2},s_{2}y_{2}\right]
\right) }
\\
\\
NG_{3}/\partial _{4}(NG_{4}\cap D_{4}) \\
\overline{\left( \left[ s_{1}x_{2},s_{2}y_{2}\right] \left[
s_{2}y_{2},s_{2}x_{2}\right] \right) } \\
\\
NG_{3}/\partial _{4}(NG_{4}\cap D_{4}) \\
\overline{\left( \left[ s_{0}x_{2},s_{2}y_{2}\right] \right) } \\
\\
NG_{3}/\partial _{4}(NG_{4}\cap D_{4}) \\
\overline{\left( \left[ s_{1}s_{0}x_{1},s_{2}y_{2}\right] \left[
s_{2}y_{2},s_{2}s_{0}x_{1}\right] \right) } \\
NG_{3}/\partial _{4}(NG_{4}\cap D_{4}) \\
\overline{\left( \left[ s_{2}s_{0}x_{1},s_{1}y_{2}\right] \left[
s_{1}y_{2},s_{2}s_{1}x_{1}\right] \left[ s_{2}s_{1}x_{1},s_{2}y_{2}\right] %
\left[ s_{2}y_{2},s_{2}s_{0}x_{1}\right] \right) } \\
\\
NG_{3}/\partial _{4}(NG_{4}\cap D_{4}) \\
\overline{\left( \left[ s_{0}y_{2},s_{2}s_{1}x_{1}\right] \left[
s_{2}s_{1}x_{1},s_{1}y_{2}\right] \right) \left[ s_{2}y_{2},s_{2}s_{1}x_{1}%
\right] }%
\end{array}%
\end{equation*}
(The elements denoted by $\overline{[\text{ },\text{ }]}$ are cosets in $%
NG_{3}/\partial _{4}(NG_{4}\cap D_{4})$ and given by the elements in $%
NG_{3}. $)
\end{proposition}

\begin{proof}
Appendix A
\end{proof}

\begin{theorem}
The category of $3$-crossed modules is equivalent to the category of
simplicial groups with Moore complex of length $3$.
\end{theorem}

\begin{proof}
Let $\mathbf{G}$ be a simplicial group with Moore complex of length
$3$. In the above proposition we showed that the group complex
\begin{equation*}
NG_{3}\overset{\partial _{3}}{\longrightarrow }NG_{2}\overset{\partial _{2}}{%
\longrightarrow }NG_{1}\overset{\partial _{1}}{\longrightarrow
}NG_{0}
\end{equation*}%
is a $3$-crossed module. Since the Moore coomplex is of length $3$, $%
NG_{4}\cap D_{4}=1,$ so $\partial _{4}(NG_{4}\cap D_{4})=1$. Thus we
can take $NG_{3}$ instead of $NG_{3}/\partial _{4}(NG_{4}\cap
D_{4})$). Finally there is a functor\
\begin{equation*}
\mathbf{\Im }_{3}\mathbf{:SimpGrp}_{\leq 3}\mathbf{\longrightarrow X}_{3}%
\mathbf{Mod}
\end{equation*}%
from the category of simplicial groups with Moore complex of length
$3$ to the category of $3$-crossed modules. Conversely, let
\begin{equation*}
K\overset{\partial _{3}}{\longrightarrow }L\overset{\partial _{2}}{%
\longrightarrow }M\overset{\partial _{1}}{\longrightarrow }N
\end{equation*}%
be a $3$-crossed module. Let $H_{0}=N$. By the action of $N$ on $M$
obtain the $H_{1}=M\rtimes N$ semidirect product. For $(m,n)\in
M\rtimes N,$ define the degeneracy and face maps as,
\begin{equation*}
\begin{array}{llll}
d_{0}: & M\rtimes N & \longrightarrow & N \\
& (m,n) & \longmapsto & n \\
d_{1}: & M\rtimes N & \longrightarrow & N \\
& (m,n) & \longmapsto & (\partial _{1}(m))n \\
s_{0}: & N & \longrightarrow & M\rtimes N \\
& n & \longmapsto & (1,n).%
\end{array}%
\end{equation*}%
Now by the actions of $M$ and $N$ on $L$ we obtain the
$H_{2}=(L\rtimes M)\rtimes (M\rtimes N)$ semidirect product. For
$l\in L,m,m^{\prime }\in M,n\in N,$ define the degeneracy and face
maps as,
\begin{equation*}
\begin{array}{llll}
d_{0}: & (L\rtimes M)\rtimes (M\rtimes N) & \longrightarrow &
(M\rtimes N)
\\
& (l,m,m^{\prime },n) & \longmapsto & (m^{\prime },n) \\
d_{1}: & (L\rtimes M)\rtimes (M\rtimes N) & \longrightarrow &
(M\rtimes N)
\\
& (l,m,m^{\prime },n) & \longmapsto & (mm^{\prime },n) \\
d_{2}: & (L\rtimes M)\rtimes (M\rtimes N) & \longrightarrow &
(M\rtimes N)
\\
& (l,m,m^{\prime },n) & \longmapsto & (\partial _{2}(l)m,\partial
_{1}(m^{\prime })n) \\
s_{0}: & (M\rtimes N) & \longrightarrow & (L\rtimes M)\rtimes
(M\rtimes N)
\\
& (m^{\prime },n) & \longmapsto & (1,1,m^{\prime },n) \\
s_{1}: & (M\rtimes N) & \longrightarrow & (L\rtimes M)\rtimes
(M\rtimes N)
\\
& (m^{\prime },n) & \longmapsto & (1,m^{\prime },1,n).%
\end{array}%
\end{equation*}%
Since $\left\{ ,\right\} _{(2)(1)}$ is a 2-crossed module\textbf{\ }there%
\textbf{\ }is an action of $L$ on $K$ defined as
\begin{equation*}
^{l}k=\left\{ \partial _{3}k,l\right\} _{(2)(1)}k^{-1}
\end{equation*}%
for $l\in L$, $k\in K$. Using this action we obtain a semidirect product $%
K\rtimes L$. The action of $(l,m)\in L\rtimes M$ on $(k,l)\in $
$K\rtimes L$ can be expressed as,
\begin{eqnarray*}
^{(1,m)}(k,l^{\prime }) &=&(^{m}(^{1}k),^{m}(^{1}l^{\prime })) \\
&=&(^{m}(k),^{m}(l^{\prime }))
\end{eqnarray*}%
\begin{eqnarray*}
^{(l,1)}(k,l^{\prime }) &=&(^{1}(^{l}k),^{1}(^{l}l^{\prime })) \\
&=&(^{l}k,^{l}l^{\prime }) \\
&=&(\text{ }^{\partial _{2}l}k\left\{ l,\partial _{3}k\right\}
_{(2)(1)},ll^{\prime }l^{-1}).
\end{eqnarray*}%
After these definitions we have the semidirect product
\begin{equation*}
H_{3}=(K\rtimes L)\rtimes (L\rtimes M)\rtimes (M\rtimes N)
\end{equation*}%
Define the degeneracy and face maps as:
\begin{equation*}
\begin{array}{llll}
d_{0}: & (K\rtimes L)\rtimes (L\rtimes M)\rtimes (M\rtimes N) &
\longrightarrow & (L\rtimes M)\rtimes (M\rtimes N) \\
& (k,l,l^{\prime },m,m^{\prime },n) & \longmapsto & (l^{\prime
},m,m^{\prime
},n) \\
d_{1}: & (K\rtimes L)\rtimes (L\rtimes M)\rtimes (M\rtimes N) &
\longrightarrow & (L\rtimes M)\rtimes (M\rtimes N) \\
& (k,l,l^{\prime },m,m^{\prime },n) & \longmapsto & (l,m,m^{\prime },n) \\
d_{2}: & (K\rtimes L)\rtimes (L\rtimes M)\rtimes (M\rtimes N) &
\longrightarrow & (L\rtimes M)\rtimes (M\rtimes N) \\
& (k,l,l^{\prime },m,m^{\prime },n) & \longmapsto & (ll^{\prime
},m,m^{\prime },n) \\
d_{3}: & (K\rtimes L)\rtimes (L\rtimes M)\rtimes (M\rtimes N) &
\longrightarrow & (L\rtimes M)\rtimes (M\rtimes N) \\
& (k,l,l^{\prime },m,m^{\prime },n) & \longmapsto & (\partial
_{3}kl,\partial _{2}l^{\prime }m,m^{\prime },n) \\
s_{0}: & (L\rtimes M)\rtimes (M\rtimes N) & \longrightarrow &
(K\rtimes
L)\rtimes (L\rtimes M)\rtimes (M\rtimes N) \\
& (l,m,m^{\prime },n) & \longmapsto & (1,l,1,m,m^{\prime },n) \\
s_{1}: & (L\rtimes M)\rtimes (M\rtimes N) & \longrightarrow &
(K\rtimes
L)\rtimes (L\rtimes M)\rtimes (M\rtimes N) \\
& (l,m,m^{\prime },n) & \longmapsto & (1,1,l,m,m^{\prime },n) \\
s_{2}: & (L\rtimes M)\rtimes (M\rtimes N) & \longrightarrow &
(K\rtimes
L)\rtimes (L\rtimes M)\rtimes (M\rtimes N) \\
& (l,m,m^{\prime },n) & \longmapsto & (1,l,1,m,m^{\prime },n)%
\end{array}%
\end{equation*}%
Thus we have a 3-truncated simplicial group $\mathbf{H}%
=\{H_{0},H_{1},H_{2},H_{3}\}.$ Applying the 3-skeleton functor
defined in section 2.1 to 3-truncation gives us a simplicial group
which will again be denoted \textbf{H} and the result has Moore
complex
\begin{equation*}
\ker \partial _{3}\rightarrow K\rightarrow L\rightarrow M\rightarrow
N.
\end{equation*}%
We set $\mathbf{H}^{\prime }=\mathbf{st}_{3}\mathbf{H}$ and note $%
NH_{p}^{\prime }=D_{p}\cap NH_{p},$ where $D_{p}$ is the subgroup of
$H_{p}$ generated by the degenerate elements, and so $NH_{p}^{\prime
}=1$ if $p>4.$ We
claim $NH_{4}^{^{\prime }}=1$. By Theorem B: case $n=4$ (see \cite{amut4}), $%
\partial _{4}(NH_{4}\cap D_{4})$ is the product of commutators. A direct
calculation using the descriptions of the actions and the face maps
above shows that these are all trivial, so $\partial _{4}(NH_{4}\cap
D_{4})=1,$ but $\partial _{4}^{\mathbf{H}}$ is a monomorphism so
$NH_{4}^{\prime }$ is trivial as required.
\end{proof}

\begin{proposition}
Let $\mathbf{G}$ be a simplicial group, let $\pi _{n}^{\prime }$ be
the homotopy groups of its $3$-crossed module and let $\pi _{n}$ be
the homotopy groups of the classifying space of $\mathbf{G}$, then
we have $\pi _{n}\cong \pi _{n}^{\prime }$ for $n=0,1,2,3,4$.
\end{proposition}

\begin{proof}
Let $\mathbf{G}$ be a simplicial group. The $n$th homotopy groups of $%
\mathbf{G}$ is the $n$th homology of the Moore complex of $\mathbf{G}$, i.e.,%
\begin{equation*}
\pi _{n}(\mathbf{G})\cong H_{n}(\mathbf{NG})\cong \frac{\ker
d_{n-1}^{n-1}\cap NG_{n-1}}{d_{n}^{n}(NG_{n})}
\end{equation*}%
Thus the homotopy groups $\pi _{n}(\mathbf{G)=}$ $\pi _{n}$ of
$\mathbf{G}$ are
\begin{equation*}
\pi _{n}=\left\{
\begin{array}{lll}
NG_{0}/d_{1}(NG_{1}) &  & n=1 \\
\dfrac{\ker d_{n-1}^{n-1}\cap NG_{n-1}}{d_{n}^{n}(NG_{n})} &  & n=2,3,4 \\
0 &  & n=0\text{ or }n>4%
\end{array}%
\right.
\end{equation*}%
and the homotopy groups $\pi _{n}^{\prime }$ of its $3$-crossed
module are
\begin{equation*}
\pi _{n}^{\prime }=\left\{
\begin{array}{lll}
NG_{0}/\partial _{1}(M) &  & n=1 \\
\ker \partial _{1}/Im(\partial _{2}) &  & n=2 \\
\ker \partial _{2}/Im(\partial _{3}) &  & n=3 \\
\ker \partial _{3} &  & n=4 \\
0 &  & n=0\text{ or }n>4.%
\end{array}%
\right.
\end{equation*}%
The isomorphism $\pi _{n}\cong \pi _{n}^{\prime }$ can be shown by a
direct calculation.
\end{proof}

\subsection{Crossed $3$-cubes}

Crossed squares (or crossed $2$-cubes) were introduced by D.
Guin-Wal\'{e}ry and J.-L. Loday\ \cite{wl}, see also \cite{jl} and
\cite{bl1}.

\begin{definition}
A crossed square is a commutative diagram of group morphisms%
\begin{equation*}
\text{$\xymatrix@R=40pt@C=60pt{L \ar[d]%
_{u} \ar[r]^-{f} & M \ar[d]_{v} \\
N \ar[r]_{g} & P  }$}
\end{equation*}%
\noindent with action of $P$ on every other group and a function
$h:M\times N\rightarrow L$ such that

\begin{enumerate}
\item[(1)] the maps $f$ and $u$ are $P$-equivariant and $g$, $v$, $v\circ f$
and $g\circ u$ are crossed modules,

\item[(2)] $f\circ h(x,y)=x^{g(y)}x^{-1}$, $u\circ h(x,y)=^{v(x)}yy^{-1}$,

\item[(3)] $h(f(z),y)=z^{g(y)}z^{-1}$, $\ h(x,u(z))=^{v(x)}zz^{-1}$,

\item[(4)] $h(xx^{\prime },y)=^{v(x)}h(x^{\prime },y)h(x,y)$, $\
h(x,yy^{\prime })=h(x,y)^{g(y)}h(x,y^{\prime })$,

\item[(5)] $h(^{t}x,^{t}y)=^{t}h(x,y)$

\end{enumerate}
\noindent for $x,x^{\prime }\in M$, $y,y^{\prime }\in N$, $z\in L$
and $t\in P$.
\end{definition}

It is a consequence of the definition that $f:L\rightarrow M$ and $%
u:L\rightarrow N$ are crossed modules where $M$ and $N$ act on $L$
via their images in $P$. A crossed square can be seen as a  crossed
module in the category of crossed modules. Also crossed squares
which were related to simplicial groups in the same way crossed
modules.

A crossed square can be seen as a complex of crossed modules of
length one
and thus, Conduch\'{e} \cite{cond}, gave a direct proof from crossed squares to $2$%
-crossed modules. This construction is the following:

Let
\begin{equation*}
\text{$\xymatrix@R=40pt@C=60pt{L \ar[d]%
_{u} \ar[r]^-{f} & M \ar[d]_{v} \\
N \ar[r]_{g} & P  }$}
\end{equation*}%
be a crossed square. Then seeing the horizontal morphisms as a
complex of
crossed modules, the mapping cone of this square is a $2$-crossed module%
\[
L\overset{\partial _{2}}{\longrightarrow }M\rtimes N\overset{\partial _{1}}{%
\longrightarrow }P,
\]%
where $\partial _{2}(z)=(f(z)^{-1},u(z))$ for $z\in L,\partial
_{1}(x,y)=g(x)g(y)$ for $x\in M$ and $y\in N,$ and the Peiffer
lifting is
given by%
\[
\left\{ (x,y),(x^{\prime },y^{\prime })\right\} =h(x,yy^{\prime
}y^{-1}).
\]

Crossed squares were generalised by G. Ellis in \cite{e1,es} called
\textquotedblleft Crossed $n$-cubes\textquotedblright\ which was
related to simplicial groups by T. Porter in \cite{p6}. Here we only
consider this
construction for $n=3$ and look at the relation between crossed $3$-cubes (see Appendix B) and $3$%
-crossed modules.

Let
\begin{equation*}
K\overset{\partial _{3}}{\longrightarrow }L\overset{\partial _{2}}{%
\longrightarrow }M\overset{\partial _{1}}{\longrightarrow }N
\end{equation*}%
be a $3$-crossed module and let $\mathbf{G}$ be the corresponding
simplicial group. The crossed $3$-cube associated to $\mathbf{G}$
defined by T.Porter in \cite{p6} is, up to a canonical isomorphism
\begin{equation*}
\text{$\xymatrix@R=25pt@C=25pt{
  & \text{ker}d_{0}^{2}\cap \text{ker}d_{1}^{2} \ar[rr]^{\nu_{Q}} \ar'[d][dd]_{\nu_{P}}
      &  & \text{ker}d_{1}^{2} \ar[dd]^{\delta_{2}}        \\
 NG_{3}  \ar[ur]^{\lambda_{N}} \ar[rr]^->>>>>{\lambda_{M}} \ar[dd]^{\lambda_{L}}
      &  & \text{ker}d_{1}^{2}\cap \text{ker}d_{2}^{2} \ar[ur]_{\nu_{Q}} \ar[dd]_>>>>>>{\nu_{Q}} \\
  & \text{ker}d_{0}^{2} \ar'[r][rr]^{\delta_{1}}
      &  & G_{2}                \\
 \text{ker}d_{0}^{2}\cap \text{ker}d_{2}^{2} \ar[rr]_{\nu_{Q}} \ar[ur]_{\nu_{P}}
      &  & \text{ker}d_{2}^{2} \ar[ur]_{\delta_{3}}     }$
}
\end{equation*}%
\noindent where $\lambda _{L}$ $,\lambda _{M},$ $\lambda _{N}$ are
restriction \ of $d_{3}^{3}$ and the others are inclusions \ The $h$-maps are%
\begin{eqnarray*}
&&%
\begin{array}{cccll}
h_{1} & : & \ker d_{1}^{2}\times \ker d_{0}^{2}\cap \ker d_{2}^{2} &
\rightarrow & NG_{3} \\
&  & (x,y) & \mapsto &
[s_{1}x_{2}s_{0}x_{2}^{-1},s_{2}y_{2}^{-1}s_{1}y_{2}^{-1}]%
\end{array}
\\
&&%
\begin{array}{cccll}
h_{2} & : & \ker d_{0}^{2}\times \ker d_{1}^{2}\cap \ker d_{2}^{2} &
\rightarrow & NG_{3} \\
&  & (x,y) & \mapsto & [s_{1}x_{2},s_{2}y_{2}s_{1}y_{2}^{-1}s_{0}y_{2}]%
\end{array}
\\
&&%
\begin{array}{cccll}
h_{3} & : & \ker d_{0}^{2}\cap \ker d_{1}^{2}\times \ker d_{2}^{2} &
\rightarrow & NG_{3} \\
&  & (x,y) & \mapsto & [s_{2}x_{2},s_{2}y_{2}s_{1}y_{2}^{-1}]%
\end{array}
\\
&&%
\begin{array}{cccll}
h_{7} & : & \ker d_{0}^{2}\cap \ker d_{2}^{2}\times \ker
d_{1}^{2}\cap \ker
d_{2}^{2} & \rightarrow & NG_{3} \\
&  & (x,y) & \mapsto & h_{2}(ix,y)=h_{2}(x,y)%
\end{array}
\\
&&%
\begin{array}{cccll}
h_{8} & : & \ker d_{0}^{2}\cap \ker d_{1}^{2}\times \ker
d_{0}^{2}\cap \ker
d_{2}^{2} & \rightarrow & NG_{3} \\
&  & (x,y) & \mapsto & h_{3}(x,iy)=h_{3}(x,y)%
\end{array}
\\
&&%
\begin{array}{cccll}
h_{9} & : & \ker d_{0}^{2}\cap \ker d_{1}^{2}\times \ker
d_{1}^{2}\cap \ker
d_{2}^{2} & \rightarrow & NG_{3} \\
&  & (x,y) & \mapsto & h_{3}(x,iy)=h_{3}(x,y)%
\end{array}%
\end{eqnarray*}%
and the others are commutators. (The name of the maps are given with
respect to the crossed $3$-cube definition in \cite{e1}.)\

Then in terms of the $3$-crossed module, this crossed cube can be written as%
\begin{equation*}
\text{$\xymatrix@R=25pt@C=25pt{
  & L \ar[rr]^{\nu_{Q}} \ar'[d][dd]_{\nu_{P}}
      &  & \overline{L \rtimes M} \ar[dd]^{\delta_{2}}        \\
 K  \ar[ur]^{\lambda_{N}} \ar[rr]^->>>>>{\lambda_{M}} \ar[dd]^{\lambda_{L}}
      &  & \overline{\overline{L}} \ar[ur]_{\nu_{Q}} \ar[dd]_>>>>>>{\nu_{Q}} \\
  & L \rtimes M \ar'[r][rr]^{\delta_{1}}
      &  & (L \rtimes M) \rtimes (M \rtimes N)                \\
 \overline{L} \ar[rr]_{\nu_{Q}} \ar[ur]_{\nu_{P}}
      &  & \overline{\overline{L \rtimes M}} \ar[ur]_{\delta_{3}}     }$}
\end{equation*}%
\noindent where

$$ L\ltimes M \cong  \{(l,m,1,1):l\in L,m\in M \},$$

$$
\overline{L\ltimes M}  =  \{(l,m,m^{\prime },1):l\in L,m,m^{\prime
}\in M,mm^{\prime }=1,l\in L,m\in M \},
$$

$$
\overline{\overline{L\ltimes M}}  =  \{(l,m,m^{\prime },n):\partial
_{2}(l)=1,\partial _{1}(m^{\prime })n=1,l\in L,m\in M \},$$

$%
\begin{array}{ccc}
L & \cong & \{(l,1,1,1):l\in L\},%
\end{array}%
$
\\

$%
\begin{array}{ccl}
\overline{L} & = & \{(l,m,1,1):\partial _{2}lm=1,l\in L,m\in M\} \\
& = & \{(l,\partial _{2}(l^{-1}),1,1):l\in L\},%
\end{array}%
$
\\

$%
\begin{array}{ccl}
\overline{\overline{L}} & = & \{(l,m,m^{\prime },1):mm^{\prime
}=1,\partial
_{2}(l)m=1,\partial _{1}(m^{\prime })n=1,l\in L,m\in M,n\in N\} \\
& = & \{(l,\partial _{2}(l^{-1}),\partial _{2}(l),1):l\in L\}.%
\end{array}%
$

By definition 2.7 given in \cite{cond} we have the mapping cone of
this
crossed $3$-cube as%
\begin{equation*}
K\rightarrow (L\ltimes \overline{\overline{L}})\ltimes \overline{L}%
\rightarrow (\overline{L\ltimes M})((L\ltimes M)\ltimes (\overline{\overline{%
L\ltimes M}}))\rightarrow (L\ltimes M)\ltimes (M\ltimes N).
\end{equation*}

\begin{example}
Let $K\overset{\partial _{3}}{\rightarrow }L\overset{\partial _{2}}{%
\rightarrow }M\overset{\partial _{1}}{\rightarrow }N$ be a
$3$-crossed module. If $M=\{1\},$ then for $i=1,2,3$ the commutative
diagram
\end{example}

\begin{equation*}
\begin{array}{cc}
&  \\
C_{i} & = \\
&
\end{array}%
\begin{array}{cc}
&  \\
\left(
\def\objectstyle{\scriptstyle}
\def\labelstyle{\scriptstyle}
\vcenter{\xymatrix@R=50pt@C=50pt{
  K \ar[r]^{\partial_{3}} \ar[d]_{\partial'_{3}}  & L  \ar[d]^{Id}   \\
 L  \ar[r]_{Id} & L  }} \right) &  \\
&
\end{array}%
\end{equation*}

\noindent is a crossed square with the following $h_{i}$-maps \
\begin{equation*}
\begin{tabular}{ll}
$h_{1}=$ & $\left\{ ,\right\} _{(2)(1)}:L\times L\rightarrow K$ \\
$h_{2}=$ & $\left\{ ,\right\} _{(0)(2)}:L\times L\rightarrow K$%
\end{tabular}%
\end{equation*}%
\begin{equation*}
\begin{tabular}{ll}
$h_{3}=$ & $\left\{ ,\right\} _{(0)(1)}:L\times L\rightarrow K$ \\
& $(x,y)\mapsto \left\{ y,x\right\} _{(1)(0)}^{-1}$%
\end{tabular}%
\end{equation*}%
where the action of $L$ on itself is by conjugation.

Since $M=\left\{ 1\right\} ,$ $\partial _{2}(l)=1_{M}$ for all $l\in
L$ .
Thus from the $3$-crossed module axioms\textbf{\ }we find%
\begin{equation*}
\begin{tabular}{lll}
$\left\{ l,\partial _{3}k\right\} _{(2)(1)}$ & $=$ & $(l\cdot k)k^{-1}$ \\
$\left\{ ll^{\prime },l^{\prime \prime }\right\} _{(2)(1)}$ & $=$ &
$l\cdot \left\{ l^{\prime },l^{\prime \prime }\right\}
_{(2)(1)}\left\{ l,l^{\prime
\prime }\right\} _{(2)(1)}$ \\
$\left\{ l,l^{\prime }l^{\prime \prime }\right\} _{(2)(1)}$ & $=$ &
$\left\{ l,l^{\prime }\right\} _{(2)(1)}\text{ }l^{\prime }\left\{
l,l^{\prime \prime
}\right\} _{(2)(1)}$ \\
$\partial _{3}\left\{ l,l^{\prime }\right\} _{(2)(1)}$ & $=$ & $%
l(^{l^{\prime }}l^{-1})$ \\
$(\left\{ l^{\prime },l\right\} _{(1)(0)})^{-1}$ & $=$ & $\left\{
l,l^{\prime }\right\} _{(2)(1)}$ \\
$\left\{ l,l^{\prime }\right\} _{(0)(2)}$ & $=$ & $1$%
\end{tabular}%
\end{equation*}%
for all $l,l^{\prime },l^{\prime \prime }\in L,$ $k,k^{\prime }\in
K$. Using these equalities and the $3$-crossed module axioms crossed
square conditions can be easily verified.

In this example the result is trivial for the $h$-map $h_{1}$ from \cite%
{conduche} since $\left\{ ,\right\} _{(2)(1)}$ is a 2-crossed
module. Here the liftings $\left\{ ,\right\} _{(0)(2)}$, $\left\{
,\right\} _{(0)(1)}$ are not 2-crossed modules but the associated
$h$-maps are crossed squares.

\begin{example}
In the universal cube definition given in \cite{e1}, take
$P,R,N,M=L,$ $S=M$ and $T_{0}=K\otimes K\otimes K$. Then

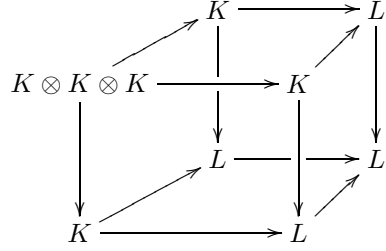
\begin{center}
$\xymatrix@R=15pt@C=15pt{
  & K \ar[rr] \ar'[d][dd]
      &  & L \ar[dd]        \\
 K \otimes K \otimes K \ar[ur]\ar[rr]\ar[dd]
      &  & K \ar[ur]\ar[dd] \\
  & L \ar'[r][rr]
      &  & L                \\
 K \ar[rr]\ar[ur]
      &  & L \ar[ur]        }$
\end{center}%

\noindent is a universal crossed 3-cube with the crossed squares
obtained by the Peiffer maps $\left\{ ,\right\} _{(0)(1)},$ $\left\{
,\right\} _{(2)(1)}, $ $\left\{ ,\right\} _{(0)(2)}$ given in the
above proposition.
\end{example}

\section{Appendix A}

\textbf{Proof of proposition 2}:

\textbf{3CM1) }%
\begin{eqnarray*}
\overline{\partial }_{3}\left( \left\{ x_{2},y_{2}\right\}
_{(2)(1)}\right)
&=&\left[ x_{2},y_{2}\right] \left[ y_{2},s_{1}\partial _{2}x_{2}\right] \\
&=&x_{2}y_{2}x_{2}^{-1}y_{2}^{-1}y_{2}s_{1}\partial
_{2}x_{2}y_{2}^{-1}s_{1}\partial _{2}x_{2}^{-1} \\
&=&x_{2}y_{2}x_{2}^{-1}(^{\partial _{2}x_{2}}y_{2})^{-1}
\end{eqnarray*}%
\textbf{\ }Since
\begin{equation*}
d_{4}(F_{(1)(3,2)}(x_{3},y_{2}))=\left[ s_{1}d_{3}x_{3},s_{2}y_{2}\right] %
\left[ s_{2}y_{2},s_{2}d_{3}x_{3}\right] \left[
x_{3},s_{2}y_{2}\right]
\end{equation*}%
we find
\begin{eqnarray*}
\{\overline{\partial }_{3}x_{3},y_{2}\}_{(2)(1)}
&=&[s_{1}\overline{\partial
}_{3}x_{3},s_{2}y_{2}][s_{2}y_{2},s_{2}\overline{\partial }_{3}x_{3}] \\
\ &\equiv &[x_{3},s_{2}y_{2}]\text{ \ \ \ }mod\text{ }\partial
_{4}(NG_{4}\cap D_{4}) \\
&=&x_{3}(^{y_{2}}x_{3})^{-1}
\end{eqnarray*}%
\textbf{\ }Since
\begin{equation*}
d_{4}(F_{(3,1)(2)}(x_{2},y_{3}))=\left[ s_{1}x_{2},s_{2}d_{3}y_{3}\right] %
\left[ s_{2}d_{3}y_{3},s_{2}x_{2}\right] \left[ s_{2}x_{2},y_{3}\right] %
\left[ y_{3},s_{1}x_{2}\right]
\end{equation*}%
we find
\begin{eqnarray*}
\left\{ x_{2},\overline{\partial }_{3}y_{3}\right\} _{(2)(1)}
&=&\left[
s_{2}x_{2},s_{2}\overline{\partial }_{3}y_{3}\right] \left[ s_{2}\overline{%
\partial }_{3}y_{3},s_{1}x_{2}\right] \\
\text{ } &\equiv &\left[ s_{2}x_{2},y_{3}\right] \left[ y_{3},s_{1}x_{2}%
\right] \text{ \ \ \ }mod\text{ }\partial _{4}(NG_{4}\cap D_{4}) \\
&=&^{x_{2}}y_{3}s_{1}x_{2}y_{3}^{-1}s_{1}x_{2}^{-1} \\
\text{ } &\equiv &^{x_{2}}y_{3}(^{\partial
_{2}x_{2}}y_{3})^{-1}\text{ \ \ \ }mod\text{ }\partial
_{4}(NG_{4}\cap D_{4})
\end{eqnarray*}%
\textbf{\ }%
\begin{eqnarray*}
\left\{ x_{2}y_{2},z_{2}\right\} _{(2)(1)} &=&\left[
s_{1}(x_{2}y_{2}),s_{2}z_{2}\right] \left[ s_{2}z_{2},s_{2}(x_{2}y_{2})%
\right] \\
&=&s_{1}(x_{2}y_{2})s_{2}z_{2}s_{1}(x_{2}y_{2})^{-1}s_{2}(x_{2}y_{2})s_{2}z_{2}^{-1}s_{2}(x_{2}y_{2})^{-1}
\\
&\equiv
&s_{2}(x_{2}y_{2})s_{2}z_{2}^{-1}s_{2}(x_{2}y_{2})^{-1}s_{1}(x_{2}y_{2}) \\
&&s_{2}z_{2}s_{1}(x_{2}y_{2})^{-1}\text{ \ }mod\text{ }\partial
_{4}(NG_{4}\cap D_{4}) \\
&=&\left\{ x_{2},y_{2}z_{2}y_{2}^{-1}\right\} _{(2)(1)}\text{
}^{\partial _{1}x_{2}}\left\{ y_{2},z_{2}\right\} _{(2)(1)}
\end{eqnarray*}%
\textbf{\ }%
\begin{eqnarray*}
\left\{ x_{2},y_{2}z_{2}\right\} _{(2)(1)} &=&\left[
s_{1}(x_{2}),s_{2}(y_{2}z_{2})\right] \left[ s_{2}(y_{2}z_{2}),s_{2}(x_{2})%
\right] \\
&\equiv &\left[ s_{2}(x_{2}),s_{2}(y_{2}z_{2})\right] \left[
s_{2}(y_{2}z_{2}),s_{1}(x_{2})\right] \\
&&(s_{2}(x_{2})s_{2}(y_{2})s_{2}(x_{2})^{-1})s_{1}(x_{2})s_{2}(y_{2})s_{1}(x_{2})^{-1}%
\text{ \ }mod\text{ }\partial _{4}(NG_{4}\cap D_{4}) \\
&=&(x_{2}y_{2}x_{2}^{-1})\cdot \left\{ x_{2},z_{2}\right\}
_{(2)(1)}\left\{ x_{2},y_{2}\right\} _{(2)(1)}
\end{eqnarray*}

\textbf{3CM2) }Since
\begin{eqnarray*}
d_{4}(F_{(3,2,0)(1)}(x_{1},y_{3})) &=&\left[ s_{2}s_{0}x_{1},s_{1}d_{3}y_{3}%
\right] \left[ s_{1}d_{3}y_{3},s_{2}s_{1}x_{1}\right] \left[
s_{2}s_{1}x_{1},s_{2}d_{3}y_{3}\right] \\
&&\left[ s_{2}d_{3}y_{3},\right] \left[ s_{2}s_{0}x_{1},y_{3}\right]
\left[ y_{3},s_{2}s_{1}x_{1}\right]
\end{eqnarray*}%
\begin{eqnarray*}
d_{4}(F_{(3,1,0)(2)}(x_{1},y_{3})) &=&\left[ s_{1}s_{0}x_{1},s_{2}d_{3}y_{3}%
\right] \left[ s_{2}d_{3}y_{3}\ ,s_{2}s_{0}x_{1}\right] \\
&&\left[ s_{2}s_{0}x_{1},y_{3}\right] \left[
y_{3},s_{1}s_{0}x_{1}\right]
\end{eqnarray*}%
and%
\begin{equation*}
d_{4}(F_{(2,1,0)(3)}(x_{1},y_{3}))=\left[ s_{2}s_{1}s_{0}d_{1}x_{1},y_{3}%
\right] \left[ y_{3},s_{1}s_{0}x_{1}\right]
\end{equation*}%
we find
\begin{equation*}
\{x_{1},\overline{\partial }_{3}y_{3}\}_{(1,0)(2)}=\{x_{1},\overline{%
\partial }_{3}y_{3}\}_{(2,0)(1)}\{\overline{\partial }_{3}y_{3},x_{1}%
\}_{(0)(2,1)}y_{3}(\text{ }^{\partial _{1}x_{1}}y_{3})^{-1}
\end{equation*}

\textbf{3CM5)} Since
\begin{eqnarray*}
d_{4}(F_{(3,0)(2,1)}) &=&[s_{0}x_{2},s_{2}s_{1}\partial
_{2}y_{2}][s_{2}s_{1}\partial
_{2}y_{2},s_{1}x_{2}][s_{2}x_{2},s_{2}s_{1}\partial _{2}y_{2}] \\
&&[s_{1}y_{2},s_{2}x_{2}][s_{1}x_{2},s_{1}y_{2}][s_{1}y_{2},s_{0}x_{2}]
\end{eqnarray*}%
we find
\begin{eqnarray*}
\{x_{2},\partial _{2}y_{2}\}_{(0)(2,1)}
&=&[s_{0}x_{2},s_{2}s_{1}\partial _{2}y_{2}][s_{2}s_{1}\partial
_{2}y_{2},s_{1}x_{2}][s_{2}x_{2},s_{2}s_{1}\partial _{2}y_{2}] \\
\ &\equiv
&[s_{1}y_{2},s_{2}x_{2}][s_{1}x_{2},s_{1}y_{2}][s_{1}y_{2},s_{0}x_{2}]\text{
\ \ \ \ \ \ \ }mod\text{ }\partial _{4}(NG_{4}\cap D_{4}) \\
&=&(\{y_{2},x_{2}\}_{(1)(2)})^{-1}\{x_{2},y_{2}\}_{(1)(0)}
\end{eqnarray*}

\textbf{3CM6) }Since
\begin{eqnarray*}
d_{4}(F_{(2,0)(3,1)}(x_{2},y_{2}))
&=&[s_{2}s_{0}d_{2}x_{2},s_{1}y_{2}][s_{1}y_{2},s_{2}s_{1}d_{2}x_{2}] \\
&&[s_{2}s_{1}d_{2}x_{2},s_{2}y_{2}][s_{2}y_{2},s_{2}s_{0}d_{2}x_{2}] \\
&&[s_{0}x_{2},s_{2}y_{2}][s_{2}y_{2},s_{1}x_{2}] \\
&&[s_{1}x_{2},s_{1}y_{2}][s_{1}y_{2},s_{0}x_{2}]
\end{eqnarray*}%
we find
\begin{eqnarray*}
\{\partial _{2}x_{2},y_{2}\}_{(2,0)(1)} &=&[s_{2}s_{0}\partial
_{2}x_{2},s_{1}y_{2}][s_{1}y_{2},s_{2}s_{1}\partial _{2}x_{2}] \\
&&[s_{2}s_{1}\partial
_{2}x_{2},s_{2}y_{2}][s_{2}y_{2},s_{2}s_{0}\partial
_{2}x_{2}] \\
&\equiv &[s_{0}x_{2},s_{2}y_{2}][s_{2}y_{2},s_{1}x_{2}] \\
&=&\{x_{2},y_{2}\}_{(0)(2)}^{-1}\text{ }^{[y_{2},x_{2}]}(\{x_{2},y_{2}%
\}_{(2)(1)})\{x_{2},y_{2}\}_{(1)(0)}
\end{eqnarray*}

\textbf{3CM7) }Since
\begin{equation*}
d_{4}(F_{(1,0)(3,2)}(x_{2},y_{2}))=[s_{1}s_{0}\overline{\partial }%
_{2}x_{2},s_{2}y_{2}][s_{2}y_{2},s_{2}s_{0}\overline{\partial }%
_{2}x_{2}][s_{0}x_{2},s_{2}y_{2}]
\end{equation*}%
we find
\begin{eqnarray*}
\{\partial _{2}x_{2},y_{2}\}_{(1,0)(2)} &=&[s_{1}s_{0}\overline{\partial }%
_{2}x_{2},s_{2}y_{2}][s_{2}y_{2},s_{2}s_{0}\overline{\partial }_{2}x_{2}] \\
&\equiv &[s_{0}x_{2},s_{2}y_{2}]\text{ \ \ \ \ \ }mod\text{
}\partial
_{4}(NG_{4}\cap D_{4}) \\
&=&\{x_{2},y_{2}\}_{(0)(2)}^{-1}
\end{eqnarray*}

\textbf{3CM8)}%
\begin{eqnarray*}
\overline{\partial }_{3}(\left\{ x_{2},y_{2}\right\} _{(1)(0)})
&=&\left[
x_{2},y_{2}\right] \left[ \overline{\partial }_{3}s_{1}x_{2},\overline{%
\partial }_{3}s_{1}y_{2}\right] \left[ \overline{\partial }_{3}s_{1}y_{2},%
\overline{\partial }_{3}s_{0}x_{2}\right] \\
&=&\left[ x_{2},y_{2}\right] s_{1}\left[ \partial _{2}x_{2},\partial
_{2}y_{2}\right] \left[ s_{1}\partial _{2}y_{2},s_{0}\partial _{2}x_{2}%
\right] \\
&=&\left[ x_{2},y_{2}\right] \left\{ \partial _{2}x_{2},\partial
_{2}y_{2}\right\}
\end{eqnarray*}

\textbf{3CM9)}%
\begin{equation*}
\overline{\partial }_{3}\left( \left\{ x_{2},y_{2}\right\} _{(0)(2)}\right) =%
\overline{\partial }_{3}\left( \left\{ \partial
_{2}x_{2},y_{2}\right\} _{(1,0)(2)}\right) ^{-1}
\end{equation*}

\textbf{3CM10)}%
\begin{eqnarray*}
\overline{\partial }_{3}\left\{ x_{2},y_{1}\right\} _{(0)(2,1)} &=&\overline{%
\partial }_{3}\left( \left[ s_{2}s_{1}y_{1},s_{2}x_{2}\right] \left[
s_{1}x_{2},s_{2}s_{1}y_{1}\right] \left[
s_{2}s_{1}y_{1},s_{0}x_{2}\right]
\right) \\
&=&\left[ s_{1}y_{1},x_{2}\right] \left[ \overline{\partial }%
_{3}s_{1}x_{2},s_{1}y_{1}\right] \left[ s_{1}y_{1},\overline{\partial }%
_{3}s_{0}x_{2}\right] \\
&=&^{y_{1}}x_{2}x_{2}^{-1}\left\{ \partial _{2}x_{2},y_{1}\right\}
\end{eqnarray*}

\textbf{3CM11) }Since
\begin{eqnarray*}
\overline{\partial }_{3}\left\{ x_{1},y_{2}\right\} _{(1,0)(2)}
&=&\left[
s_{0}x_{1},y_{2}\right] \left[ y_{2},\overline{\partial }_{3}s_{1}s_{0}x_{1}%
\right] \\
\overline{\partial }_{3}\left\{ x_{1},y_{2}\right\}
_{(1,0)(2)}\left[ \overline{\partial
}_{3}s_{1}s_{0}x_{1},y_{2}\right] &=&\left[ s_{0}x_{1},y_{2}\right]
\end{eqnarray*}%
\begin{eqnarray*}
\overline{\partial }_{3}\left\{ x_{1},y_{2}\right\} _{(2,0)(1)}
&=&\left[ s_{0}x_{1},y_{2}\right] \left[ y_{2},s_{1}x_{1}\right]
\left[ s_{1}x_{1},\partial _{3}s_{1}y_{2}\right] \left[ \partial
_{3}s_{1}y_{2},s_{0}x_{1}\right] \\
&=&\left[ s_{0}x_{1},y_{2}\right] \left[ y_{2},s_{1}x_{1}\right]
\left[ s_{1}x_{1},s_{1}\partial _{2}y_{2}\right] \left[
s_{1}\partial
_{2}y_{2},s_{0}x_{1}\right] \\
&=&\left[ s_{0}x_{1},y_{2}\right] \left[ y_{2},s_{1}x_{1}\right]
\left\{ x_{1},\partial _{2}y_{2}\right\}
\end{eqnarray*}%
we find
\begin{eqnarray*}
\overline{\partial }_{3}\left\{ x_{1},y_{2}\right\} _{(2,0)(1)} &=&\overline{%
\partial }_{3}\left\{ x_{1},y_{2}\right\} _{(1,0)(2)}\left[ \overline{%
\partial }_{3}s_{1}s_{0}x_{1},y_{2}\right] \left[ y_{2},s_{1}x_{1}\right]
\left\{ x_{1},\partial _{2}y_{2}\right\} \\
&=&\overline{\partial }_{3}\left\{ x_{1},y_{2}\right\}
_{(1,0)(2)}\left[ s_{1}s_{0}\partial _{1}x_{1},y_{2}\right] \left[
y_{2},s_{1}x_{1}\right]
\left\{ x_{1},\partial _{2}y_{2}\right\} \\
&=&\overline{\partial }_{3}\left\{ x_{1},y_{2}\right\} _{(1,0)(2)}\text{ }%
^{\partial _{1}x_{1}}y_{2}\text{ }^{x_{1}}y_{2}\left\{
x_{1},\partial _{2}y_{2}\right\}
\end{eqnarray*}

\textbf{3CM12)} Since
\begin{equation*}
d_{4}(F_{(0)(3,1)}(x_{3},y_{2}))=[s_{0}d_{3}x_{3},s_{1}y_{2}][s_{1}y_{2},s_{1}d_{3}x_{3}][s_{2}d_{3}x_{3},s_{2}y_{2}][s_{2}y_{2},x_{3}]
\end{equation*}%
we find
\begin{eqnarray*}
\{\overline{\partial }_{3}x_{3},y_{2}\}_{(1)(0)}
&=&[s_{0}\overline{\partial
}_{3}x_{3},s_{1}y_{2}][s_{1}y_{2},s_{1}\overline{\partial }_{3}x_{3}][s_{2}%
\overline{\partial }_{3}x_{3},s_{2}y_{2}] \\
\ &\equiv &[s_{2}y_{2},x_{3}]\text{ \ \ \ \ \ \ \ \ }mod\text{
}\partial
_{4}(NG_{4}\cap D_{4}) \\
&=&(^{y_{2}}x_{3})x_{3}^{-1}
\end{eqnarray*}%
\textbf{\ }Since
\begin{eqnarray*}
d_{4}(F_{(3,0)(1)}(x_{2},y_{3})) &=&\left[ s_{0}x_{2},s_{1}d_{3}y_{3}\right] %
\left[ s_{1}d_{3}y_{3},s_{1}x_{2}\right] \\
&&\left[ s_{2}x_{2},s_{2}d_{3}y_{3}\right] \left[
y_{3},s_{2}x_{2}\right]
\end{eqnarray*}%
we find
\begin{eqnarray*}
\left\{ x_{2},\overline{\partial }_{3}y_{3}\right\} _{(1)(0)}
&=&\left[
s_{0}x_{2},s_{1}\overline{\partial }_{3}y_{3}\right] \left[ s_{1}\overline{%
\partial }_{3}y_{3},s_{1}x_{2}\right] \left[ s_{2}x_{2},s_{2}\overline{%
\partial }_{3}y_{3}\right] \\
\text{ } &\equiv &\left[ y_{3},s_{2}x_{2}\right] \text{ \ \ \ \ \ \ \ \ }mod%
\text{ }\partial _{4}(NG_{4}\cap D_{4}) \\
&=&y_{3}(^{x_{2}}y_{3})^{-1}
\end{eqnarray*}

\textbf{3CM14) }Since
\begin{equation*}
d_{4}(F_{(0)(3,2)}(x_{3},y_{2}))=[s_{0}d_{3}x_{3},s_{2}y_{2}]
\end{equation*}%
we find
\begin{eqnarray*}
\{\overline{\partial }_{3}x_{3},y_{2}\}_{(0)(2)}
&=&[s_{0}\overline{\partial
}_{3}x_{3},s_{2}y_{2}] \\
\ &\equiv &1\text{ \ }mod\text{ }\partial _{4}(NG_{4}\cap D_{4})
\end{eqnarray*}

\textbf{3CM15)} Since \textbf{\ }%
\begin{equation*}
d_{4}(F_{(3,0)(2)}(x_{2},y_{3}))=\left[ s_{0}x_{2},s_{2}d_{3}y_{3}\right] %
\left[ y_{3},s_{0}x_{2}\right]
\end{equation*}%
and
\begin{equation*}
d_{4}(F_{(1,0)(2)}(x_{2},y_{3}))=\left[ s_{1}s_{0}\partial
_{2}x_{2},s_{2}\partial _{3}y_{3}\right] \left[ s_{2}\partial
_{3}y_{3},s_{2}s_{0}\partial _{2}x_{2}\right] \left[
s_{0}x_{2},y_{3}\right]
\end{equation*}%
we find
\begin{eqnarray*}
\left\{ x_{2},\overline{\partial }_{3}y_{3}\right\} _{(0)(2)}
&=&\left[
s_{0}x_{2},s_{2}\overline{\partial }_{3}y_{3}\right] \\
\text{ } &\equiv &\left[ y_{3},s_{0}x_{2}\right] \text{ \ \ \ \ \
}mod\text{
}\partial _{4}(NG_{4}\cap D_{4}) \\
\text{ } &\equiv &\left[ s_{2}s_{0}\partial _{2}x_{2},s_{2}\overline{%
\partial }_{3}y_{3}\right] \left[ s_{2}\overline{\partial }%
_{3}y_{3},s_{1}s_{0}\partial _{2}x_{2}\right] \text{ \ \ \ }mod\text{ }%
\partial _{4}(NG_{4}\cap D_{4})\text{\ \ } \\
&=&\left\{ \partial _{2}\left( x_{2}\right) ,\overline{\partial
}_{3}\left( y_{3}\right) \right\} _{(1,0)(2)}^{-1}
\end{eqnarray*}

\textbf{3CM16) }Since
\begin{eqnarray*}
d_{4}(F_{(2,0)(1)}(x_{2},y_{3})) &=&\left[ s_{2}s_{0}\partial
_{2}x_{2},s_{1}\partial _{3}y_{3}\right] \left[ s_{1}\partial
_{3}y_{3},s_{2}s_{1}\partial _{2}x_{2}\right] \\
&&\left[ s_{2}s_{1}\partial _{2}x_{2},s_{2}\partial _{3}y_{3}\right]
\left[
s_{2}\partial _{3}y_{3},s_{2}s_{0}\partial _{2}x_{2}\right] \\
&&\left[ s_{0}x_{2},y_{3}\right] \left[ y_{3},s_{1}x_{1}\right]
\end{eqnarray*}%
we find
\begin{eqnarray*}
\left\{ \partial _{2}x_{2},\overline{\partial }_{3}y_{3}\right\}
_{(2,0)(1)}
&=&\left[ s_{2}s_{0}\partial _{2}x_{2},s_{1}\overline{\partial }_{3}y_{3}%
\right] \left[ s_{1}\overline{\partial }_{3}y_{3},s_{2}s_{1}\partial
_{2}x_{2}\right] \\
&&\left[ s_{2}s_{1}\partial _{2}x_{2},s_{2}\overline{\partial }_{3}y_{3}%
\right] \left[ s_{2}\overline{\partial }_{3}y_{3},s_{2}s_{0}\partial
_{2}x_{2}\right] \\
\text{ } &\equiv &\left[ s_{0}x_{2},y_{3}\right] \left[ y_{3},s_{1}x_{1}%
\right] \text{ \ \ \ \ }mod\text{ }\partial _{4}(NG_{4}\cap D_{4}) \\
\text{ } &\equiv &\left[ s_{0}x_{2},y_{3}\right] y_{3}(^{\partial
_{2}x_{2}}y_{3})^{-1}\text{ \ \ }mod\text{ }\partial _{4}(NG_{4}\cap
D_{4})
\\
\text{ } &\equiv &\left[ s_{1}s_{0}\partial _{2}x_{2},s_{2}\overline{%
\partial }_{3}y_{3}\right] \left[ s_{2}\overline{\partial }%
_{3}y_{3},s_{2}s_{0}\partial _{2}x_{2}\right] \text{ \ \ \ \ }mod\text{ }%
\partial _{4}(NG_{4}\cap D_{4}) \\
&=&\left\{ \partial _{2}x_{2},\overline{\partial }_{3}y_{3}\right\}
_{(1,0)(2)}
\end{eqnarray*}

\textbf{3CM17) }Since%
\begin{eqnarray*}
d_{4}(F_{(0)(2,1)}(x_{3},y_{2})) &=&\left[
s_{0}d_{3}x_{3},s_{2}s_{1}d_{2}y_{2}\right] \left[
s_{2}s_{1}d_{2}y_{2},s_{1}d_{3}x_{3}\right] \\
&&\left[ s_{2}d_{3}x_{3},s_{2}s_{1}d_{2}y_{2}\right] \left[ s_{1}y_{2},x_{3}%
\right]
\end{eqnarray*}%
and
\begin{equation*}
d_{4}(F_{(0)(2,1)}(x_{2},y_{3}))=\left[ s_{2}s_{1}d_{2}x_{2},y_{3}\right] %
\left[ y_{3},s_{1}x_{2}\right]
\end{equation*}%
we find
\begin{eqnarray*}
\left\{ \overline{\partial }_{3}x_{3},\partial _{2}y_{2}\right\}
_{(0)(2,1)}
&=&\left[ s_{0}\overline{\partial }_{3}x_{3},s_{2}s_{1}\partial _{2}y_{2}%
\right] \left[ s_{2}s_{1}\partial _{2}y_{2},s_{1}\overline{\partial }%
_{3}x_{3}\right] \\
&&\left[ s_{2}\overline{\partial }_{3}x_{3},s_{2}s_{1}\partial _{2}y_{2}%
\right] \\
\text{ } &\equiv &\left[ s_{1}y_{2},x_{3}\right] \text{ \ \ \ \ }mod\text{ }%
\partial _{4}(NG_{4}\cap D_{4}) \\
\text{ } &\equiv &\left[ x_{3},s_{2}s_{1}\partial _{2}y_{2}\right]
\text{ \
\ \ \ }mod\text{ }\partial _{4}(NG_{4}\cap D_{4}) \\
&=&x_{3}(^{\partial _{2}y_{2}}x_{3})^{-1}
\end{eqnarray*}

\textbf{3CM18)}
\begin{eqnarray*}
\partial _{2}\{x_{1},y_{1}\} &=&[x_{1},y_{1}][y_{1},\partial _{2}s_{0}x_{1}]
\\
&=&x_{1}y_{1}x_{1}^{-1}(\partial _{1}x_{1y_{1}})^{-1}
\end{eqnarray*}


\section{Appendix B}

\subsection{Crossed 3-cube}

Given a commutative diagram of groups

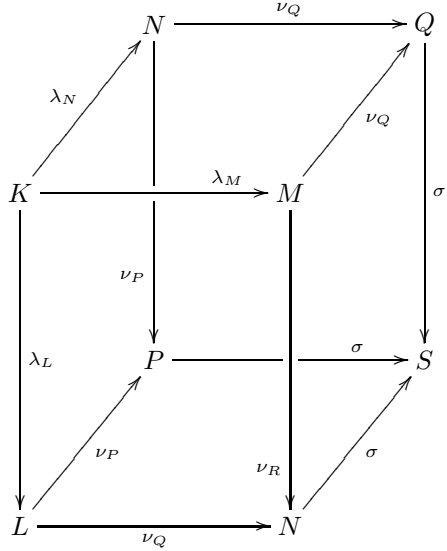
\begin{equation*}
\text{$\xymatrix@R=50pt@C=35pt{
  & N \ar[rr]^{\nu_{Q}} \ar'[d][dd]_{\nu_{P}}
      &  & Q \ar[dd]^{\sigma}        \\
 K  \ar[ur]^{\lambda_{N}} \ar[rr]^->>>>>{\lambda_{M}} \ar[dd]^{\lambda_{L}}
      &  & M \ar[ur]_{\nu_{Q}} \ar[dd]_>>>>>>{\nu_{R}} \\
  & P \ar'[r][rr]^{\sigma}
      &  & S                \\
 L \ar[rr]_{\nu_{Q}} \ar[ur]_{\nu_{P}}
      &  & N \ar[ur]_{\sigma}     }$}
\end{equation*}

\noindent in which there is a group action of $S$ on each of the
other seven groups (hence the eight groups act on each other via the
action of $S$), and there are six functions
\begin{eqnarray*}
&&%
\begin{array}{lllll}
h_{1} & : & Q\times L & \rightarrow & K%
\end{array}
\\
&&%
\begin{array}{lllll}
h_{2} & : & P\times M & \rightarrow & K%
\end{array}
\\
&&%
\begin{array}{lllll}
h_{3} & : & N\times R & \rightarrow & K%
\end{array}
\\
&&%
\begin{array}{lllll}
h_{4} & : & P\times R & \rightarrow & L%
\end{array}
\\
&&%
\begin{array}{lllll}
h_{5} & : & Q\times R & \rightarrow & M%
\end{array}
\\
&&%
\begin{array}{lllll}
h_{6} & : & P\times Q & \rightarrow & N%
\end{array}%
\end{eqnarray*}%
we say that this structure is a crossed $3$-cube of groups if

\begin{enumerate}
\item each of the nine squares

$\xymatrix@R=5pt@C=5pt{
  K \ar[rr]  \ar[dd]  && L  \ar[dd] && K \ar[rr]  \ar[dd]  && M  \ar[dd] && K \ar[rr]  \ar[dd]  && R  \ar[dd] && L \ar[rr]  \ar[dd]  && R  \ar[dd] & & M \ar[rr]  \ar[dd]  && R  \ar[dd] & & N \ar[rr]  \ar[dd]  && O  \ar[dd]  \\ \\
 Q  \ar[rr] && S && P \ar[rr] && S && N \ar[rr] && S && P  \ar[rr] && S  && Q \ar[rr] && S && P  \ar[rr] && S   }$

$\xymatrix@R=5pt@C=5pt{
  K \ar[rr]  \ar[dd]  && M  \ar[dd] && K \ar[rr]  \ar[dd]  && L  \ar[dd] && K \ar[rr]  \ar[dd]  && M  \ar[dd]  \\ \\
 L  \ar[rr] && R && N \ar[rr] && P && N \ar[rr] && Q  }$

\noindent is a crossed square; for the last three squares the functions $%
h:L\times M\rightarrow K,$ $h:N\times L\rightarrow K,$ $h:N\times
M\rightarrow K$ are respectively given by $h(l,m)=h(v_{P}l,n),$ $%
h(n,l)=h(n,v_{R}l),$ $h(n,m)=h((n,v_{R}m)$

\item $h((v_{P}n)(v_{P}l),m)h((v_{Q}m)(v_{Q}n),l)=h(n,(v_{R}l)(v_{R}m))$

\item $^{q}h(h(p,q^{-1})^{-1},r)=$ $^{p}h(q,h(p^{-1},r))$ $%
^{r}h(p,h(q,r^{-1})^{-1})$

\item \ \

$%
\begin{array}{ccc}
\lambda _{L}h(p,m) & = & h(p,v_{R}m) \\
\lambda _{L}h(n,r) & = & h(v_{P}n,r) \\
\lambda _{M}h(q,l) & = & h(q,v_{R}l) \\
\lambda _{M}h(n,r) & = & h(v_{Q}n,r) \\
\lambda _{N}h(p,m) & = & h(p,v_{Q}m) \\
\lambda _{N}h(q,l) & = & h(v_{P}l,q)^{-1}%
\end{array}%
$

\item \ \

$%
\begin{array}{ccc}
h(v_{Q}m,l) & = & h(v_{P}l,m)^{-1} \\
h(n,v_{R}l) & = & h(v_{Q}n,l) \\
h(n,v_{R}m) & = & h(v_{P}n,m)%
\end{array}%
$
\end{enumerate}

for all $l\in L,$ $m\in M,$ $n\in N,$ $p\in P,$ $q\in Q,$ $r\in R$.

\section*{Acknowledgements}

This works was partially supported by T\"{U}B\.{I}TAK (The Scientific and Technical Research Council of Turkey). \\
Project  Number : 107T542

\newpage

\author{Zekeriya Arvas\.{i}\\             
zarvasi@ogu.edu.tr      
\\ 
         Osmangazi University\\        
         Department of Mathematics and Computer Sciences \\
         Art and Science Faculty \\
        Eski\c{s}ehir/Turkey }
\bigskip

\author{Tufan Sait Kuzpinari   \\         
stufan@dpu.edu.tr     
\\    
         Aksaray University\\
         Department of Mathematics\\
         Art and Science Faculty\\
         Aksaray/Turkey}
\bigskip

\author{Enver \"{O}nder Uslu \\           
euslu@aku.edu.tr       
\\    
         Afyon Kocatepe University\\
         Department of Mathematics\\
         Art and Science Faculty \\
         Afyonkarahisar/Turkey}


\begin{thebibliography}{99}
\bibitem{patron3} \textsc{Z. Arvasi} and \textsc{T.Porter}, Higher
dimensional Peiffer elements in simplicial commutative algebras, \emph{%
Theory and Applications of Categories}, 3, 1-23, (1997).

\bibitem{zarvasi} \textsc{Z. Arvasi} and \textsc{E. Ulualan}, On algebraic
models for homotopy 3-types, \emph{Journal of Homotopy and Related Structures%
}, Vol. \textbf{1(1)}, 1-27, (2006).

\bibitem{AA} \textsc{I. Ak\c{c}a} and \textsc{Z. Arvas\.{i}}, \textrm{%
Simplicial and crossed Lie algebras}, \emph{Homology, Homotopy and
Applications}, Vol. \textbf{4}, No.1, 43-57, (2002).

\bibitem{baues} \textsc{H.J. Baues}, Combinatorial homotopy and
4-dimensional complexes, \emph{Walter de Gruyter}, (1991).

\bibitem{baues2} \textsc{H.J. Baues}, Homotopy types, \emph{Handbook of Algebraic Topology},
Edited by I. M. James, 1-72, (1995).

\bibitem{bourn} \textsc{D. Bourn}, Moore normalization and Dold-Kan theorem
for semiabelian categories, Proceedings of the conference Categories
in Algebra, Geometry and Mathematical Physics, \emph{Contemporary
Mathematics} vol.431, July 2007.

\bibitem{brown} \textsc{R. Brown} and \textsc{N.D. Gilbert}, \textrm{%
Algebraic models of 3-types and automorphism structures for crossed
modules}, \emph{Proc. London Math.Soc.}, (\textbf{3}), \textbf{59},
51-73, (1989).

\bibitem{bl1} \textsc{R. Brown} and \textsc{J.-L. Loday}, Van Kampen
theorems for diagram of spaces, \emph{Topology}, \ \textbf{26},
311-335, (1987).

\bibitem{c} \textsc{P. Carrasco}, Complejos hipercruzados, cohomologia y
extensiones, \emph{Ph.D. Thesis}, \ Universidad de Granada, (1987).

\bibitem{cc} \textsc{P. Carrasco} and \textsc{A.M. Cegarra}, Group-theoretic
algebraic models for homotopy types, \ \emph{Journal of  Pure and
Applied Algebra}, \ \textbf{75}, 195-235, (1991).

\bibitem{ladra} \textsc{J.L. Castiglioni} and \textsc{M. Ladra}, \textrm{%
Peiffer elements in simplicial groups and algebras}, \emph{Journal
of Pure and Applied Algebra}, \ \textbf{212}, 2115-2128, (2008).

\bibitem{conduche} \textsc{D. Conduch\'{e}}, Modules crois\'{e}s g\'{e}n\'{e}%
ralis\'{e}s de longueur 2, \emph{Journal of  Pure and Applied
Algebra}, \ \textbf{34}, 155-178, (1984) .

\bibitem{cond} \textsc{D. Conduch\'{e}}, Simplicial crossed modules and
mapping cones, \emph{Georgian Mathematical Journal}, \ \textbf{ 10},%
 623-636, (2003).

\bibitem{curtis} \textsc{E.B. Curtis}, Simplicial homotopy theory, \ \emph{%
Adv. in Math.}, \textbf{6}, 107-209, (1971).

\bibitem{duskin} \textsc{J. Duskin}, Simplicial methods and the
interpretation of triple cohomology, \emph{Memoir A.M.S.}, \ Vol. 3 \ \textbf{%
163}, (1975).

\bibitem{e1} \textsc{G.J. Ellis}, Crossed modules and their higher
dimensional analogues, \ \emph{Ph.D. Thesis}, \ U.C.N.W. \ (1984).

\bibitem{es} \textsc{G.J. Ellis} and \textsc{R.Steiner},\ Higher dimensional
crossed modules and the homotopy groups of (n+1)-ads, \
\emph{Journal of  Pure and Applied Algebra}, \ \textbf{46}, 117-136,
(1987).

\bibitem{glenn} \textsc{P.Glenn}, \ Realization of cohomology classes in
arbitrary exact categories, \ \emph{Journal of  Pure and Applied
Algebra}, \ \textbf{25}, 33-107, (1982).

\bibitem{wl} \textsc{D.Guin-Wal\'{e}ry} and \textsc{J.-L. Loday},
Obstructions \`{a} l'excision en K-th\'{e}orie alg\'{e}brique, \emph{%
Springer Lecture Notes in Math.,} \ \textbf{854}, 179-216, (1981).

\bibitem{jl} \textsc{J.-L. Loday}, Spaces with finitely many non-trivial
homotopy groups, \emph{Journal of Pure and Applied Algebra}, \
\textbf{24}, 179-202,\ (1982).

\bibitem{may} \textsc{J.P. May}, Simplicial objects in algebraic topology, \
\emph{Van Nostrand, Math. Studies 11.}

\bibitem{amut4} \textsc{A. Mutlu} and \textsc{T. Porter}, Applications of
Peiffer pairing in the Moore complex of a simplicial group,
\emph{Theory and Applications of Categories}, Volume 4, No. 7,
148-173 (1998).

\bibitem{p6} \textsc{T. Porter}, n-Types of simplicial groups and crossed
n-cubes, \emph{Topology}, \textbf{32}, 5-24, (1993).

\bibitem{robert} \textsc{D. M. Roberts} and \textsc{U. Schreiber}, The Inner automorphism 3-group of a strict
2-group, \emph{Journal of Homotopy and Related Structures}, Vol. 3,
193-244, (2008).

\bibitem{jh1} \textsc{J.H.C. Whitehead}, Combinatorial homotopy I and
II, \emph{Bull. Amer. Math. Soc.}, \ \textbf{55}, 213-245 and
496-543, (1949).
\end{thebibliography}
\end{document}